\documentclass[12pt,letterpaper,english]{amsart}
\usepackage{lmodern}
\usepackage{helvet}

\usepackage[T1]{fontenc}
\usepackage[latin9]{inputenc}
\usepackage{lmodern}
\usepackage{mathrsfs}
\usepackage{amsthm}
\usepackage{amstext,amsfonts,mathtools}
\usepackage{amssymb}
\usepackage{tikz,tikz-cd}
\usepackage[all]{xy}
\usepackage{array}
\usepackage{bm}
\usepackage{color}
\usepackage{esint}
\usepackage{graphicx}
\usepackage{cancel}
\usepackage{enumitem}
\usepackage{setspace}

\makeatletter
\pdfpageheight\paperheight
\pdfpagewidth\paperwidth

\newcommand\cA{{\mathcal A}}

\newcommand\cC{{\mathcal C}}
\newcommand\cD{{\mathcal D}}

\newcommand\cF{{\mathcal F}}
\newcommand\cG{{\mathcal G}}

\newcommand\cL{{\mathcal L}}

\newcommand\cN{{\mathcal N}}
\newcommand\cO{{\mathcal O}}

\newcommand\cS{{\mathcal S}}

\newcommand\cU{{\mathcal U}}
\newcommand\cV{{\mathcal V}}

\newcommand\cX{{\mathcal X}}
\newcommand\cY{{\mathcal Y}}

\newcommand\fF{{\mathfrak F}}
\newcommand\fG{{\mathfrak G}}

\newcommand\fR{{\mathfrak R}}

\def\ft{\mathfrak{t}}
\def\ff{\mathfrak{f}}
\def\fg{\mathfrak{g}}
\def\fn{\mathfrak{n}}
\def\fs{\mathfrak{s}}

\newcommand\bA{{\mathbb A}}

\newcommand\CC{{\mathbb C}}

\newcommand\GG{{\mathbb G}}
\newcommand\HH{{\mathbb H}}

\newcommand\PP{{\mathbb P}}
\newcommand\QQ{{\mathbb Q}}
\newcommand\RR{{\mathbb R}}

\newcommand\VV{{\mathbb V}}

\newcommand\bV{{\mathbb V}}

\newcommand\ZZ{{\mathbb Z}}
\newcommand\BBG{{\widetilde{\Gamma}}}

\def\BV{\mathbf{V}}

\def\BE{\mathbf{E}}

\newcommand{\Extm}{\mathrm{Ext}_{\textup{MHS}}}
\newcommand{\Homm}{\mathrm{Hom}_{\textup{MHS}}}
\newcommand{\codim}{\mathrm{codim}}

\newcommand{\bs}{\backslash}
\newcommand{\sm}{\setminus}

\newcommand{\complex}{\mathbb{C}}

\newcommand{\re}[1]{\mathfrak{#1}}
\newcommand{\im}[1]{\mathfrak{#1}^*}

\def\d{\partial}
\def\e{\epsilon}

\def\gr{\mathrm{Gr}}

\def\ay{\mathbf{i}}
\def\ANF{\mathrm{ANF}}
\def\CH{\mathrm{CH}}
\newenvironment{psmatrix}
  {\left(\begin{smallmatrix}}
  {\end{smallmatrix}\right)}
\def\ac{\alpha}
\def\bc{\beta}
\def\act{\widetilde{\ac}}
\def\Res{\mathrm{Res}}
\def\TS{\widehat{S}}
\def\BS{\overline{S}}
\def\TG{\widehat{\Gamma}}
\def\sD{\mathscr{D}}
\def\acto{\widetilde{\ac_1}}
\def\actr{\widetilde{\ac_r}}

\numberwithin{equation}{section}
\numberwithin{figure}{section}
\theoremstyle{plain}
\newtheorem{thm}{\protect\theoremname}[section]
  \theoremstyle{definition}
  
  \theoremstyle{remark}
  \newtheorem{rem}[thm]{\protect\remarkname}
  \theoremstyle{plain}
  \newtheorem{prop}[thm]{\protect\propositionname}
  \theoremstyle{plain}
  \newtheorem{cor}[thm]{\protect\corollaryname}
  \theoremstyle{plain}
  \newtheorem{claim}[thm]{\protect\claimname}
  \theoremstyle{definition}
  
  \theoremstyle{plain}
  \newtheorem{lem}[thm]{\protect\lemmaname}
  \theoremstyle{plain}
  
\theoremstyle{plain}
  
\theoremstyle{plain}
\newtheorem*{thm*}{\protect\theoremname}
\theoremstyle{definition}

\theoremstyle{definition}
\newtheorem*{defn*}{\protect\definitionname}
\theoremstyle{definition}
\newtheorem*{thx}{Acknowledgments}

\makeatother

  \providecommand{\corollaryname}{Corollary}
  \providecommand{\claimname}{Claim}
  \providecommand{\definitionname}{Definition}
  \providecommand{\examplename}{Example}
  \providecommand{\lemmaname}{Lemma}
  \providecommand{\propositionname}{Proposition}
 \providecommand{\problemname}{Problem}
  \providecommand{\remarkname}{Remark}
\providecommand{\theoremname}{Theorem}
\providecommand{\conjecturename}{Conjecture}

\author{RJ Acu\~na}
\address{Shanghai Institute for Mathematics and Interdisciplinary Sciences (SIMIS), Shanghai, 200433, China; and Research Institute of Intelligent Complex Systems, Fudan University, Shanghai, 200433, China}
\email{rj@simis.cn}

\author{Devin Akman}
\address{Department of Mathematics, Washington University in St.~Louis, 1 Brookings Drive, St.~Louis, MO 63130, USA}
\email{akman@wustl.edu}

\author{Matt Kerr}
\address{Department of Mathematics, Washington University in St.~Louis, 1 Brookings Drive, St.~Louis, MO 63130, USA}
\email{matkerr@wustl.edu}

\title{Real regulator maps with finite 0-locus}

\begin{document}

\begin{abstract}
A Laurent polynomial in two variables is tempered if its edge polynomials are cyclotomic.  Variation of coefficients leads to a family of smooth complete genus $g$ curves carrying a canonical algebraic $K_2$-class over a $g$-dimensional base $S$, hence to an extension of admissible variations of MHS (or normal function) on $S$.  We prove that the $\RR$-split locus of this extension is finite.  Consequently, the torsion locus of the normal function and the $A$-factor locus for the family of curves are also finite.
\end{abstract}
\maketitle


\section{Introduction}

Hodge theory provides the basic interface between the algebraic and transcendental worlds in modern complex geometry. To a family $\pi\colon \cX\to \cS$ of smooth projective varieties, it associates a polarized variation of Hodge structure (VHS), or equivalently a period mapping \cite{Gr}.  To an algebraic cycle on such a family, homologous to zero on fibers, it associates an admissible normal function \cite{Sa}.  This can be viewed either as a holomorphic section of a bundle of generalized Jacobians, or as an extension of the original VHS by a Tate object --- the simplest kind of variation of mixed Hodge structure (VMHS).  We can distinguish two kinds of normal functions, ``classical'' and ``higher'', with the former arising from cycles in the usual Chow group of $\cX$, and the latter from higher Chow cycles in the sense of Bloch \cite{Bl}.

A \emph{special subvariety} of $\cS$ is a component of the (Hodge) locus where a given mixed variation acquires extra rational Hodge tensors, such as when a normal function takes torsion values.  Although it is known that the special locus is always a countable union of algebraic varieties \cite{CDK,BP,BPS}, questions remain about the finiteness and field of definition of these components, especially in the \emph{nonclassical} case where Griffiths transversality enters.  In particular, according to the \emph{Zilber-Pink conjecture} for VMHS \cite{BU}, the maximal \emph{atypical} special subvarieties (see \S\ref{S1.3}) should be finite in number, with tame geometry providing strong evidence in the pure case \cite{BKU}.  What those techniques have so far been unable to produce are finiteness results for \emph{0-dimensional} components of the Hodge locus.

In this article we address a compelling (and nonclassical) case of the Zilber-Pink conjecture for \emph{higher} normal functions, where the VHS of a family of curves is extended by a Tate object in weight $4$.  This sort of VMHS arises in many contexts, including $K_2$ of curve families \cite{DK}, Mahler measure of 2-variable polynomials \cite{FRV}, local mirror symmetry and Hori-Vafa models \cite{BKV}, and $A$-polynomials of knot complements \cite{GM}.  For a generically nontorsion normal function of this type, it turns out that all components of the torsion locus are atypical, and under additional assumptions, of dimension zero.  One may then argue (cf.~\S\ref{S1.3}) that Zilber-Pink predicts the torsion locus is a finite set of points.  Our main results below confirm this prediction; in fact, we prove something stronger, as we now describe in detail.

\subsection{Statement of results}

Let $S$ be a quasi-projective complex $d$-manifold,\footnote{That is, $S$ denotes the complex points of a smooth complex quasi-projective variety.} carrying a nonzero integral polarized variation of Hodge structure $\mathbf{V}=(\cV,\cF^{\bullet},\bV,\nabla,Q)$, pure of weight and level $1$.  Here $\bV$ is the underlying $\ZZ$-local system, $\cV\,(\cong \VV\otimes \cO_S)$ the underlying holomorphic vector bundle (or its sheaf of sections), $\cF^{\bullet}$ the Hodge filtration on $\cV$, $Q$ the polarization, and $\nabla\colon \cV\to \cV\otimes \Omega^1_S$ the flat connection with kernel $\bV_{\CC}:=\bV\otimes \CC$.  Denote by $g\,(\geq 1)$ the rank of $\cF^1\cV$, so that the Hodge numbers of $\cV$ are $(g,g)$.

Consider a nontorsion admissible normal function $R\in \ANF_S(\BV(2))$, or equivalently an extension
\begin{equation}\label{E1}
0\to \BV\to \BE \to \ZZ(-2)_S\to 0
\end{equation}
of admissible $\ZZ$-VMHS \cite{SZ} on $S$ which is not split as an extension of $\QQ$-VMHS.  We may represent $R$ as a global holomorphic section of the generalized Jacobian bundle of $\BV(2)$, i.e.~of the sheaf $\cV/\VV(2)=\cV/(2\pi\ay)^2\VV$, with the horizontality property $\nabla R\in H^0(S,\cF^1\cV\otimes \Omega^1_S)$.  (We will describe the consequences of admissibility later.)  Writing $\cA$ for the sheaf of real analytic functions on $S$ and $\cV_{\RR}:=\VV\otimes \cA$, we can also consider the projection $R\mapsto \fR^*:=\mathrm{Im}(R)$ to $\ay\cV_{\RR}$, called the \emph{real regulator} of the extension.\footnote{Writing $\VV_{\RR}:=\VV\otimes_{\ZZ} \RR$ and $\cV_{\CC}:=\cV_{\RR}\otimes_{\RR} \CC$, and noting $\ZZ(2)\subset \RR$, this ``projection'' is the composition $\cV/\VV_{\ZZ(2)}\twoheadrightarrow \cV/\VV_{\RR}\twoheadrightarrow \cV/\cV_{\RR}\hookrightarrow \cV_{\CC}/\cV_{\RR}\cong \ay \cV_{\RR}$.}

The normal function $R$ defines special loci in $S$:
\begin{itemize}[leftmargin=0.5cm]
\item $Z_{\ZZ}(R):=$ $\ZZ$-split locus of $\BE$ $=$ zero-locus of $R$;
\item $Z_{\QQ}(R):=$ $\QQ$-split locus of $\BE$ $=$ torsion locus of $R$; and
\item $Z_{\RR}(R):=$ $\RR$-split locus of $\BE$ $=$ zero-locus of $\fR^*$.
\end{itemize}
Clearly we have $Z_{\RR}(R)\supseteq Z_{\QQ}(R)\supseteq Z_{\ZZ}(R)$.
A fundamental result of Brosnan and Pearlstein \cite{BP} states that $Z_{\ZZ}(R)$ is an algebraic subvariety of $S$; an immediate corollary is that $Z_{\QQ}(R)=\bigcup_{n\in \ZZ_{>0}}Z_{\ZZ}(nR)$ is a countable union of algebraic subvarieties.  But we claim that in our present scenario much more is true:

\begin{thm}\label{T1}
Suppose that $d=1$, i.e.~that $S$ is a curve.  Then $Z_{\QQ}(R)$ is finite; and if $\BV$ has trivial fixed part, then $Z_{\RR}(R)$ is also finite.
\end{thm}

At any $s\in S$, we can view $\nabla R|_s$ as a map from $T_s$ to $\cF^1\cV|_s$; its rank is bounded by $\mathrm{min}(d,g)$.

\begin{thm}\label{T2}
Suppose that $d\leq g$ and $\nabla R$ has maximal rank at every point of $S$.  Then $Z_{\RR}(R)$ \textup{(}and thus $Z_{\QQ}(R)$\textup{)} is finite.
\end{thm}

\begin{rem}
As pointed out by one of the referees, these theorems are a little bit surprising because in general the real split locus is a real analytic subvariety, not a complex analytic one.  For instance, a Kummer extension $\nu_f\in\ANF_S(\ZZ(1)_S)$ corresponds to a unit $f\in \mathcal{O}^{\times}_{\text{alg}}(S)$, and $\nu_f$ is $\RR$-split where $|f|=1$.  But our hypotheses ensure that the $\RR$-split locus in our case has dimension zero.
\end{rem}

A geometric source for normal functions of the above type is that of \emph{tempered Laurent polynomials}, a notion introduced by F.~Rodriguez-Villegas \cite{FRV}; it has been applied or generalized in various contexts in \cite{DK,Ke,DKS}.  Let $\Delta\subset \RR^2$ be a polygon with integer vertices.  Label the boundary points $\d\Delta \cap \ZZ^2$ by $\ell^i$ and the $g$ interior points $\text{int}(\Delta)\cap \ZZ^2$ by $m^j$. Let 
\begin{equation}\label{eq1.2}
F_a=\sum_i c_ix_1^{\ell_1^i}x_2^{\ell_2^i}+\sum_{j=1}^g a_j x_1^{m^j_1}x_2^{m^j_2}
\end{equation}
be a family of Laurent polynomials in 2 variables, with $a_j$ free and $c_i$ constant, Newton polygon $\Delta$, and cyclotomic edge polynomials (given by the $c_i$). Write $C_a^*:=\{F_a(x_1,x_2)=0\}\subset \GG_m^2$ and $C_a:=\overline{C_a^*}\subset \PP_{\Delta}$, $S:=\bA^g\sm \Sigma$ for the locus where $C_a$ is smooth, and $\cV$ for the VHS over $S$ with fiber $H^1(C_a)$ at $a\in S$. Then some fixed multiple of the ``coordinate symbol'' $\{-x_1,-x_2\}\in \CH^2(C_a^*,2)$ lifts to a class $\xi_a\in \CH^2(C_a,2)$ for every $a\in S$. Taking images under the integral regulator/Abel-Jacobi maps
\begin{equation}\label{eq1.3}
\mathrm{AJ}\colon \CH^2(C_a,2)\to H^1(C_a,\CC/\ZZ(2))	
\end{equation}
produces $R\in \ANF_S(\cV(2))$ as above. We call this the \emph{geometric setting}.

\begin{prop}\label{P3}
In the geometric setting, $\nabla R$ has maximal rank at every point of $S$.
\end{prop}

So the hypotheses of Theorem \ref{T2} hold, and in particular we have

\begin{cor}\label{C4}
In the geometric setting, $\xi_a$ is torsion for only finitely many $a\in S$.
\end{cor}

\subsection{Outline of the paper}\label{S1.2}
We begin in \S\ref{S2} with a very explicit proof of Theorem \ref{T1} for $g=1$, to familiarize the reader with the asymptotic aspects of the method in the simplest case.    This is then generalized to arbitrary $g$ in \S\ref{S3}.  The main idea is to use transversality to constrain the differential behavior of $R$ and admissibility (together with the notions of canonical extension and singularities of normal functions) to constrain its asymptotic behavior.  For background the reader may consult \cite[\S2]{KP}.

For the proof of Theorem \ref{T2} in \S\ref{S4}, two new ingredients are needed.  The first is a classical asymptotic Hodge theory result of Cattani-Kaplan \cite{CK}, on the constancy of the monodromy weight filtration in the interior of a monodromy cone.  We use this to bound $Z_{\RR}(R)$ away from singularities of $R$.  The second comes from the theory of o-minimal structures \cite{vdDM} which has recently been highly influential in Hodge theory through works such as \cite{BBT,BKU}.  In order to deal with boundary points where $R$ is nonsingular, we make use of the $\RR_{\text{an,exp}}$-definability of functions arising from the real and imaginary parts of $R$.

In each case, the proof begins by applying a base-change to make the boundary into a normal-crossing divisor and monodromies of $\bV$ unipotent.  Since finite monodromies become trivial, this results in some boundary strata with no monodromy. (Here ``boundary'' refers to the complement of the finite cover of $S$ in its compactification.) Though the VHS and normal function\footnote{For a normal function of the type considered here (in $\ANF(\cV(2))$, for $\cV$ of level and weight $1$), one cannot have $R$ singular along a boundary stratum unless $\cV$ is singular there.  Moreover, if $R$ is singular at a boundary point, then it is in fact singular in codimension one along a divisor passing through that point.  See \S\ref{S4.1}.} can be extended across such strata, we may not have $\nabla R$ of maximal rank there and cannot expect the finiteness result to hold there either.  So these strata remain part of the boundary for the purposes of our analysis, and there are thus four local scenarios to consider (after the base change):  points at which both $R$ and $\cV$ are singular; points at which only $\cV$ is singular; boundary points at which neither is singular; and interior points. The goal is to show that about any point in the compactified cover of $S$, there is a neighborhood whose intersection with the cover of $S$ contains no point (except possibly the center) at which $\fR^*$ is zero.  

Turning to the geometric case in \S\ref{S5}, we briefly review some consequences of an explicit representation of the regulator class $R$ developed in \cite{DK,DKS}.  This is used to calculate $\nabla R$ and prove Prop.~\ref{P3}.  Finally, in Appendix \ref{AA}, we offer a more explicit analysis (generalizing that in \S\ref{S2}) to the asymptotic behavior of the real and imaginary parts of $R$ that avoids o-minimal theory.  This is an appendix because the solution seems somewhat remarkable, but not part of the main article because it only works when neither the VHS nor the normal function is singular at the boundary point in question.

\subsection{Relation to other works}\label{S1.3}
Under the hypotheses of Theorems \ref{T1}-\ref{T2}, it is straightforward to argue from \cite[Cor.~1.3]{BP} that $Z_{\QQ}(R)$ is a \emph{countable} set of points.  Moreover, some of the arguments which we phrase in terms of the properties of the integral regulator class $R$ can be rephrased as in \cite[\S4]{BP} in terms of the local normal form for the mixed variation $\BE$.

On the other hand, the conjectural framework around atypical Hodge loci \cite{Kl,BKU,BU,KlT} predicts that $Z_{\QQ}(R)$ is finite.  To explain this, we use the formulation of Zilber-Pink for $\ZZ$VMHS in \cite[Conj.~4.10]{BU}.  Associated to a mixed variation $\BE$ on $S$ is a \emph{Hodge datum} $(\cG,\cD)$, comprising its (generic) Mumford-Tate group and domain, and \emph{period map} $\Phi\colon S\to \BBG\backslash \cD$ (with $\BBG$ the image\footnote{Alternatively, one may replace $S$ by a finite cover and $\tilde{\Gamma}$ by an arithmetic subgroup of $\mathcal{G}(\QQ)$ containing the image of monodromy.} of the monodromy representation).  The special subvarieties $S'\subset S$ for $\BE$ are the irreducible components of $\Phi^{-1}(\BBG'\backslash \cD')$ for which the Hodge datum of $\BE|_{S'}$ is $(\cG',\cD')\subsetneq (\cG,\cD)$; such an $S'$ is atypical if 
\begin{equation}\label{E4}
\codim_S(S')<\codim_{\Gamma\bs\cD}(\BBG'\bs\cD')\,.
\end{equation}
The Zilber-Pink conjecture states that the atypical Hodge locus is a \emph{finite} union of \emph{maximal} such components.

Now let $\BE$, $\BV$ and $R$ be as in \eqref{E1}, where we may assume (by arguing as in \S\ref{S3.1} and deleting unnecessary factors) that $\BV$ has trivial fixed part and $R$ generically takes nontorsion values in all irreducible factors of $\BV$.  Write $(G,D)$ for the Hodge datum of $\BV$, $2g$ for its rank, and $\Gamma$ for its monodromy group.  By \cite[Cor.~2.9]{KT}, $R$ is then \emph{generic} in the sense that $\cG$ is an extension of $G$ by $\mathbb{G}_a^{2g}$, and $\BBG$ is (or enlarges to) an extension of $\Gamma$ by $\ZZ^{2g}$.  Accordingly, $\cD$ [resp.~$\BBG\backslash\cD$] is fibered over $D$ [resp.~$\Gamma\backslash D$] with fibers $\cong \CC^{2g}$ [resp.~$(\CC^*)^{2g}$].  The subdomains $\cD'\subset\cD$ corresponding to $R$ becoming torsion are the countably many sections of $\cD\to D$ taking values in $\QQ(2)^{2g}$.

For $S'$ an irreducible component of $Z_{\QQ}(R)$, this yields $\text{RHS}\eqref{E4}\geq 2g$. On the other hand, since the rank of $\nabla R$ is bounded above by $g$, \cite[Thm.~1.4]{GZ} implies that $\text{LHS}\eqref{E4}\leq g$, making $S'$ atypical.  According to Zilber-Pink, the \emph{maximal} atypical special subvarieties thus consist of finitely many components of $Z_{\QQ}(R)$, together with finitely many maximal atypical special subvarieties for $\BV$ which are not contained in $Z_{\QQ}(R)$.  Restricting $R$ to these components and inducing on dimension, one deduces that $Z_{\QQ}(R)$ has only finitely many components.

Finally, if we assume that $d\leq g$ and (if $d> 1$) that $\nabla R$ has maximal rank \emph{everywhere} on $S$, then components of $Z_{\QQ}(R)$ are (0-dimensional) points and Zilber-Pink implies that $|Z_{\QQ}(R)|<\infty$.  It is this prediction that our Theorems \ref{T1}-\ref{T2} confirm.  This contrasts with the approach via monodromic atypicality in \cite{BKU,BU}, which produces powerful results for positive-dimensional components of the atypical Hodge locus (including torsion loci of normal functions \cite{KT,GZ}), but cannot say anything about atypical special points.

There is also an intriguing connection to knot theory.  By \cite[\S4]{CCGLS} and \cite[Prop.~5.2]{GM} (see also the expository account in \cite{Ak}), the Laurent polynomial $F_a$ in \eqref{eq1.2} is a factor of an $A$-polynomial of a 3-manifold with torus boundary if and only if $\xi_a$ is torsion.  Call $F_a$ \emph{exact} if $a\in Z_{\RR}(R)$; then the exact locus contains the $A$-factor locus.  In \cite[\S4]{GM}, it was conjectured that the exact locus (hence also the $A$-factor locus) is finite in $S$, with a proof claimed for $g=1$.  However, this proof did not account for degenerations, only showing that exact polynomials are isolated in $S$. (This is insufficient since $S$ is noncompact:  see the argument at the end of \S\ref{S2}.)  In any case, Theorem \ref{T2} and Proposition \ref{P3} immediately imply the Guilloux-March\'e conjecture in generality:

\begin{cor}
In the geometric setting, $F_a$ is exact for only finitely many $a\in S$.
\end{cor}

\noindent Moreover, the second author's paper \cite{Ak} shows how to reinterpret Corollary \ref{C4}  in this context without reference to a family.  Namely, define the \emph{overgenus} of an $A$-factor to be the minimum genus of an $A$-polynomial curve containing the $A$-factor curve as an irreducible component.  Then according to [op.~cit., Thm.~4.1], there are only finitely many smooth $A$-factors with overgenus bounded by any fixed positive integer.

A brief word on notation:  for a holomorphic function $F$, we write $\fF=\mathrm{Re}(F)$ and $\fF^*=\mathrm{Im}(F)$; the same goes for holomorphic sections (or vector-valued functions), when the meaning is clear.

\begin{thx}
We thank Fran\c cois Brunault for bringing \cite{GM} to our attention, and Patrick Brosnan, Phillip Griffiths, John McCarthy, and Gregory Pearlstein for related discussions.  The authors are also grateful to the referees for very helpful comments, which raised the level of exposition.  The first author acknowledges the financial support of the Ann W.~and Spencer T.~Olin Chancellor's Graduate Fellowship; and the third author acknowledges support from the NSF via Grants DMS-2101482 and DMS-2502708, and from the Simons Foundation.  
\end{thx}

\section{Families of elliptic curves}\label{S2}

In this section we assume that $\BV$ is a weight-$1$, rank-$2$ $\ZZ$-PVHS on a curve $S$ with unipotent (possibly trivial) monodromies about points of $\Sigma=\overline{S}\sm S$, and consider an extension $\BE$ as in \eqref{E1}.  Its nonzero mixed Hodge numbers are $h^{1,0}=h^{0,1}=h^{2,2}=1$.  Let $R$ denote the corresponding admissible normal function.\footnote{Such normal functions are discussed for example in \cite[\S4]{DK} and \cite[\S4]{GKS}.}  On any simply connected open subset of $S$, $R$ has a lift to a holomorphic section $\tilde{R}$ of $\cV$.  Assuming that no such local lift is constant, we claim that any $x\in \overline{S}$ has a neighborhood $U$ such that $U\sm\{x\}\subset S$ and $\fR^*|_{U\sm\{x\}}$ is nowhere zero.  We defer other aspects of the 1-dimensional base story (such as justifying the last assumption) to \S\ref{S3}.

Identifying $U$ with a disk $\Delta=\{z\in \CC\mid |z|<\epsilon\}$, and taking $\tilde{R}$ to denote a multivalued lift on $\Delta^*$, we need to show --- after shrinking $\epsilon$ if necessary --- that $\fR^*=\mathrm{Im}(\tilde{R})$ is nowhere $0$ on $\Delta^*:=\Delta\sm\{0\}$.

The variation $\BV$ is very easy to describe on $\Delta^*$:  let $\ac,\bc$ denote a basis of a ``nearby fiber'' $V:=\bV|_{z_0}$ of the local system; we may assume that the monodromy operator $T=e^N$ sends $\ac\mapsto \ac-n\bc$ and $\bc\mapsto \bc$ for some $n\in \ZZ_{\geq 0}$.  (Dually we think of $V^{\vee}=\ZZ\langle \alpha^{\vee},\beta^{\vee}\rangle$ as $H_1$ of an elliptic curve, with the Dehn twist sending $\beta^{\vee}\mapsto \beta^{\vee}+n\alpha^{\vee}$.)  Write $\ell(z):=\tfrac{1}{2\pi\ay}\log(z)$ and $\act:=\ac+n\ell(z)\bc$, which is a \emph{single valued} section of $\cV$.  The canonical extension $\cV_e$ of $\cV$ to $\Delta$ is $\cO\langle \act,\bc\rangle$; this comes with a logarithmic connection $\nabla\colon \cV_e\to \cV_e\otimes \Omega^1_{\Delta}\langle \log \{0\}\rangle$.  A holomorphic generator of $\cF^1_e$ takes the form
\begin{equation}\label{E2.1}
\omega=\ac+\tau(z)\bc=\act+\eta(z)\bc=\ac+(\eta(z)+n\ell(z))\bc,
\end{equation}
where the multivalued function $\tau$ should be thought of as the period ratio and $\eta(0)$ as indicating the LMHS $V_{\lim}:=\psi_z\BV$ at the origin.

Now by Griffiths transversality and admissibility for $\BE$, $\tfrac{1}{2\pi\ay}\nabla_{z\d_z}\tilde{R}$ extends to a holomorphic section of $\cF^1_e$ on $\Delta$.  Let $a$ be the value at $0$ of its $\act$-coefficient; this is an integer, the \emph{singularity class} of $R$.  It is worth explaining its meaning briefly: the LMHS $E_{\lim}=\psi_z\BE$ has mixed Hodge numbers $h^{0,0}=h^{1,1}=h^{2,2}=1$.  The action of the extended monodromy logarithm (also denoted $N$) on $E_{\lim}$ is the same thing as $-2\pi\ay\Res_0(\nabla)$.  The class $\tfrac{-1}{(2\pi\ay)^2}\tilde{R}$ may be seen as the difference of Hodge and integral lifts of a generator of $\ZZ(-2)$, and so the $(-1,-1)$-morphism $N$ sends this generator in $E_{\lim}^{2,2}$ to $a\omega(0)$ in $E_{\lim}^{1,1}$ (equiv., $a\act(0)$ in $\gr^W_2 E_{\lim}$).
If $a$ is nonzero, then we can write on $\Delta^*$ (really on its universal cover)
\begin{equation}\label{E2.2}
\tilde{R}=(2\pi\ay a\log(z)+\text{holomorphic})\ac+(\cdots)\bc\,.
\end{equation}
Taking the imaginary part gives $\fR^*=(2\pi a\log|z|+\text{harmonic})\ac+(\cdots)\bc$, whence $\fR^*$ is nonzero on $\Delta^*$ (shrinking the disk as needed).

Henceforth we assume that $a=0$, i.e.~that $R$ is \emph{nonsingular} at the origin.  In this case, $N$ is zero on $E_{\lim}^{2,2}$ and $\nabla_{z\d_z}\tilde{R}$ vanishes at $0$.  That is, we can write $\nabla_{\d_z}\tilde{R}=:f\omega$ on $\Delta^*$, where $f$ extends to a holomorphic function on $\Delta$. (By the nonconstancy assumption, $f$ is not identically zero.)  Indeed, if $R$ is nonsingular,\footnote{See \cite[\S2.12]{KP} for a general discussion of asymptotic invariants of normal functions; the main point here is that if the singularity class vanishes, then the limit is defined in the Jacobian of the kernel of $N$.} then $R$ has a single-valued lift $\tilde{R}$ to $\cV_e$, with \emph{limit} $\tilde{R}(0)\in \ker(N)\subseteq V_{\lim}$.  Thus if $n>0$ (i.e.~$\cV$ is singular at $0$), then $\tilde{R}(0)=\Psi \bc$ for some $\Psi\in \CC$, since $N\act=-n\bc\notin \ker(N)$.  If $n=0$, then $\tilde{R}(0)=\Psi_{\alpha}\ac+\Psi_{\beta}\bc$ is arbitrary.

Set $k=\text{ord}_0(f)$, $F:=\int_0 f\,dz=z^{k+1}U$, $\tilde{F}:=\int_0F\tfrac{dz}{z}=z^{k+1}\tilde{U}$, and $G:=\int_0\eta f\,dz=FH_0$, where $U,\tilde{U},H_0$ are holomorphic and $U,\tilde{U}$ are units (on $\Delta$).  From this point on, the argument is different for $n=0$ and $n>0$.

If $n>0$, integrating $f\omega$ gives
\begin{align*}
\tilde{R}&= \int_0 \{f\ac+	(\eta f+n\ell(z)f)\bc\}dz\\
&=z^{k+1}U\ac+\{\Psi+z^{k+1}(UH_0+nU\ell(z)-\tfrac{n}{2\pi\ay}\tilde{U})\}\bc\\
&=z^{k+1}U\ac+\{\Psi+z^{k+1}U(H+n\ell(z))\}\bc,
\end{align*}
where $H:=H_0-\tfrac{n}{2\pi\ay}\tfrac{\tilde{U}}{U}$ is holomorphic.
Taking imaginary parts and writing $\psi:=\text{Im}(\Psi)$ gives
\begin{align*}
\fR^*:=\text{Im}(\tilde{R})=\text{Im}(z^{k+1}U)\ac+\{\psi+\text{Im}[z^{k+1}U(H+n\ell(z))]\}\bc.
\end{align*}
If $\psi\neq 0$, then $\fR^*$ is nonzero\footnote{Always after possibly shrinking $\Delta$ in this argument.} on $\Delta$ and we are done.  If $\psi=0$, then
\begin{align*}
\fR^*&=\text{Im}(z^{k+1}U)\act+\{\text{Im}[z^{k+1}U(H+n\ell(z))]-n\ell(z)\text{Im}(z^{k+1}U)\}\bc\\
&=\text{Im}(z^{k+1}U)\act+\{\text{Im}(z^{k+1}UH)-\tfrac{n}{2\pi}\log|z|\bar{z}^{k+1}\bar{U}\}\bc\\
&=\text{Im}(z^{k+1}U)\act-\tfrac{n}{2\pi}\bar{z}^{k+1}\log|z|U\left\{1-\frac{2\pi}{n}\frac{\text{Im}(z^{k+1}UH)}{\bar{z}^{k+1}\log|z|\bar{U}}\right\}\bc.
\end{align*}
Now $|H|\leq B$ for some $B>0$ on $\Delta$.  Shrinking $\Delta$ so that $|z|<e^{-4\pi B/n}$, we have $$\frac{2\pi}{n}\left|\frac{\text{Im}(z^{k+1}UH)}{\bar{z}^{k+1}\log|z|\bar{U}}\right|\leq \frac{2\pi}{n}\left|\frac{z^{k+1}UH}{\bar{z}^{k+1}\bar{U}}\right|\frac{1}{-\log|z|}\leq \frac{2\pi B/n}{-\log|z|}<\frac{1}{2},$$
and thus the coefficient of $\bc$ is nonzero on $\Delta^*$.

If $n=0$, then $\eta=\tau$ is nonzero, and integration by parts shows that\footnote{Note that this is different from the $H$ in the last paragraph.} $H:=H_0/\tau=1+\mathcal{O}(z)$ is a holomorphic unit.  We thus have $G=FH\tau$ and $\tilde{R}=(\Psi_{\alpha}+F)\ac+(\Psi_{\beta}+G)\bc$.  If the imaginary part of either constant $\Psi_{\alpha}$ or $\Psi_{\beta}$ is nonzero, we are done; if both are zero, then\footnote{Recall that we write $\fF=\mathrm{Re}(F)$, $\fF^*=\mathrm{Im}(F)$, etc.~throughout this paper.} $\fR^*=\fF^*\ac+\fG^*\bc$.  Suppose we have a set $\mathscr{S}\subset \Delta^*$ of common zeroes of $\fF^*$ and $\fG^*$ with $0$ as limit point.  We may shrink $\Delta$ so that $\text{Im}(\tau H)$ is nonzero there.  Then on $\mathscr{S}$ we have $\fG^*=\text{Im}(\tau H)\fF+\text{Re}(\tau H)\fF^*=\text{Im}(\tau H)\fF$ so that $F$ vanishes on $\mathscr{S}$.  Since $F$ is holomorphic and $G=FH\tau$, $F$ and $G$ vanish identically on $\Delta$, contradicting the local nonconstancy of $\tilde{R}$.

Thus we have shown that every $x\in\overline{S}$ has an analytic open neighborhood $U_x$ such that $U_x{\setminus}\{x\}\subset S$ and
\begin{equation}\label{E2.3}
Z_{\RR}(R)\cap (U_x{\setminus}\{x\})=\emptyset.
\end{equation}
The $\{U_x\}_{x\in\overline{S}}$ constitute an open cover of $\overline{S}$; so by compactness of $\overline{S}$, it has a finite subcover of the form $\{U_x\}_{x\in \Sigma\cup \Sigma'}$, with $\Sigma'\subset S$ finite. By \eqref{E2.3}, we have $Z_{\RR}(R)\subseteq \Sigma'$, which gives $|Z_{\RR}(R)|<\infty$ as desired.

Though it is nice to see a very explicit analytic argument once, it would clearly become unwieldy in the most general setting.  Much of the analysis is subsumed in results of Bierstone--Milman \cite{BM} and van den Dries--Miller \cite{vdDM} in what follows.

\section{One-parameter families}\label{S3}

In this section, we give a complete proof of Theorem \ref{T1}.  

\subsection{Preliminaries}\label{S3.1}

Let $\BV$ be a $\ZZ$-PVHS with Hodge numbers $(g,g)$ on a curve $S$, and $R\in \mathrm{ANF}_S(\BV(2))$ a nontorsion normal function.

\begin{lem}\label{L3.1}
If $\BV$ has trivial fixed part, then every local lift $\tilde{R}$ of $R$ is nonconstant.
\end{lem}
\begin{proof}
Suppose otherwise: $R$ is nontorsion, but on some disk $D\subset S$, $\tilde{R}$ is the flat continuation of a vector $(2\pi\ay)^2 v \in V_{\CC}$ in the fiber of the local system $\mathbb{V}_{\CC}$ over a base point of $D$.  Now $R$ is a holomorphic section of $\mathcal{V}/\mathbb{V}_{\ZZ(2)}$, so by analytic continuation its local lifts are all flat, which forces $R$ to be a section of $\mathbb{V}_{\CC}/\mathbb{V}_{\ZZ(2)}$.  Equivalently, we have $(T-I)v\in V_{\ZZ}$ for every $T\in \Gamma\leq \mathrm{GL}(V_{\ZZ})$, where $\Gamma$ is the monodromy group of $\mathbb{V}$.  Now since $R$ is nontorsion, $v\notin V_{\QQ}$; accordingly, let $\sigma\in \mathrm{Gal}(\CC/\QQ)$ be such that $v -v^{\sigma}\neq 0$.  By triviality of the fixed part, $V_{\CC}^{\Gamma}=\{0\}$; so for \emph{some} $T\in \Gamma$ we have $T(v-v^{\sigma})\neq v-v^{\sigma}$, hence $((T-I)v)^{\sigma}\neq (T-I)v$.  This is a contradiction.
\end{proof}

Now assume that we have shown in the trivial fixed part case that $Z_{\RR}(R)$ is finite.  Dropping the fixed-part hypothesis, we claim that $Z_{\QQ}(R)$ is finite.  By the Theorem of the Fixed Part, we can decompose $\BV_{\QQ}=\BV_{\text{fix}}\oplus \BV_{\text{var}}$ as a direct sum in the category of $\QQ$-PVHS, where $\BV_{\text{fix}}$ is constant and $\BV_{\text{var}}$ has trivial fixed part; decompose $R=R_{\text{fix}}+R_{\text{var}}$ accordingly.  If $R_{\text{var}}$ is nontorsion, then by assumption $Z_{\RR}(R_{\text{var}})\supseteq Z_{\QQ}(R)$ is finite.  So suppose that $R_{\text{var}}$ is torsion, and observe that $R_{\text{fix}}$ lives in $\ANF_S(\BV_{\text{fix}}(2))$, which is an extension of $\Homm(\QQ(0),H^1(S,\VV_{\text{fix}}(2)))$ by (constant sections) $$\Extm^1(\QQ(0),H^0(S,\VV_{\text{fix}}(2)))\cong V_{\text{fix},\CC}/V_{\text{fix},\QQ(2)}.$$  But since $H^1(S,\VV_{\text{fix}}(2))\cong H^1(S)\otimes V_{\text{fix}}(2)$ has weights $-2$ and $-1$, it has no Hodge $(0,0)$-classes and the Hom-space is trivial.  Hence $R_{\text{fix}}$ is constant and  nontorsion (since $R$ is nontorsion), whence $Z_{\QQ}(R)$ is empty.

Henceforth we assume $\BV$ has trivial fixed part.
We will need a general result which allows us to pass to an appropriate finite cover, both in this and the next section.  We do not need to worry about whether this introduces a nontrivial fixed part, since nonconstant local lifts of $R$ cannot become constant on the cover.

\begin{lem}\label{L3.2}
Let $\mathbf{H}$ be a $\ZZ$-PVHS on a smooth complex quasi-projective variety $Y$.	Then there exist a finite \'etale cover $\widehat{Y}\to Y$ and smooth projective compactification $\overline{Y}\supseteq \widehat{Y}$ such that $\Sigma=\overline{Y}\sm\widehat{Y}$ is a normal crossing divisor about which the base-change $\widehat{\mathbf{H}}$ has unipotent monodromies.
\end{lem}

\begin{proof}
Write $H_{\ZZ}$ for a fiber of $\HH$, and let $\Gamma\leq \mathrm{Aut}(H_{\ZZ},Q)\,$ be the image of the monodromy representation.  First, $\mathrm{Aut}(H_{\ZZ},Q)$ has a finite-index subgroup which is \emph{neat} \cite[17.4]{Bo}, i.e.~no element has a torsion eigenvalue (root of unity different from $1$ itself).  Let $\TG$ denote its intersection with $\Gamma$.  Second, the \'etale covering $\widehat{Y^{\text{an}}}\to Y^{\text{an}}$ of complex manifolds induced by $\TG$ (or rather, its pullback to $\pi_1(Y^{\text{an}})$) is in fact algebraic by the ``Riemann existence theorem'' \cite[XII]{Gro}.  Let $\widehat{Y}$ denote the underlying variety.  Third, by resolution of singularities \cite[p.~146]{Hi}, $\widehat{Y}$ admits a smooth compactification $\overline{Y}$ with normal-crossing complement.  The monodromies of the base-change $\widehat{\mathbf{H}}$ are then unipotent by combining the monodromy theorem \cite[3.1]{Gri} with neatness of $\TG$.
\end{proof}

Let $\TS\subset \BS$ be the cover of $S$ guaranteed by Lemma \ref{L3.2}, and write $\widehat{R}$ for the base-change of $R$.  So by covering $\BS$ with disks $U\cong \Delta$ centered at every point (with $\Delta^*\subset S$), we are reduced as in \S\ref{S2} to the following scenario:
\begin{itemize}[leftmargin=0.5cm]
\item a $\ZZ$-PVHS $\BV$ on $\Delta^*$ with Hodge numbers $(g,g)$ and unipotent monodromy operator $T$; and
\item a nontorsion normal function $\widehat{R}|_{\Delta^*}\in \ANF_{\Delta^*}(\BV(2))$ with nonconstant local lifts.\footnote{This will simply be denoted by $R$ in the paragraphs that follow.}
\end{itemize}
Now assume we have shown the following:
\begin{claim}\label{C3.3}
After shrinking $\Delta$, $\mathrm{Im}(\widehat{R}|_{\Delta^*})$ is nowhere zero on $\Delta^*$.  
\end{claim}
\noindent Then compactness of $\BS$ produces a finite subcover; and since the centers of this finite set of disks must contain $Z_{\RR}(\widehat{R})\cup \Sigma$, $Z_{\RR}(\widehat{R})$ (and thus $Z_{\RR}(R)$) is finite, proving Theorem \ref{T1}.

\subsection{The asymptotic part}\label{S3.2}

In this subsection we prove Claim \ref{C3.3}.  We begin with two lemmas on real analytic functions, which will also be used in the next section.

\begin{lem}\label{L3.4}
If $A(t)$ and $B(t)$ are real analytic functions on $(-\e,\e)$, and $A(t)-B(t)\log|t|$ has zeroes converging to $t=0$, then $A$ and $B$ are identically zero.  	
\end{lem}
\begin{proof}
If one is identically zero, the other must be too by the local connectedness of analytic varieties \cite[Cor.~2.7]{BM}. If neither is identically $0$, divide into 2 cases, which amount to $u(t)=t^clog|t|$ ($c>0$) and $t^c u(t)=\log|t|$ ($c\geq 0$) with $u$ an analytic unit. Neither has solutions on $(-\e,0)\cup (0,\e)$ for sufficiently small $\e>0$. Contradiction.
\end{proof}

\begin{lem}\label{L3.5}
Let $A_1,\ldots,A_p$ be real analytic functions on a neighborhood $\cU$ of $0$ in $\RR^n$, and set $X:=\{x\in \cU\mid A_i(x)=0,\;i=1,\ldots,p\}$.  Let $W\subset \cU$ be another closed analytic subset containing $0$, and set $X^-:=X\setminus X\cap W$.  Then either (i) it is possible to shrink $\cU$ so that $X^-$ becomes empty, or (ii) $X^-$ contains a connected real analytic \emph{manifold} $Y$ of positive dimension.
\end{lem}
\begin{proof}
By \cite[Prop.~2.10]{BM},\footnote{Also see Theorem 3.14 and the beginning of the proof of Theorem 5.1 in [op.~cit.].} after shrinking $\cU$, $X$ has a \emph{finite} stratification into connected semianalytic sets $X_i$ which are analytic submanifolds of $\cU$ of pure dimensions $\delta_i$.  So after further shrinking $\cU$, the only possible $0$-dimensional stratum is the origin.  If some $X_i$ is not contained in $W$, then $\delta_i>0$ and $X_i\cap W$ is a closed analytic subset, and we take $Y$ to be a connected component of $X_i\setminus X_i\cap W$.
\end{proof}

At this point it will suffice to work with rational coefficients.  Let $V$ be a fiber of the local system $\bV_{\QQ}$ on $\Delta^*$, and $\{\ac_1,\ldots,\ac_g;\bc_1,\ldots,\bc_g\}$ denote a symplectic basis.  We may assume that the monodromy logarithm about $z=0$ takes the form $N=\begin{psmatrix}0 & 0 \\ -\mathcal{N} & 0\end{psmatrix}$ (with $\mathcal{N}=\mathcal{N}_{ij}$ a $g\times g$ block).  Define $\widetilde{\ac_i}:=\ac_i+\sum_j \ell(z)\mathcal{N}_{ij}\bc_j$ to get (together with $\bc_1,\ldots,\bc_g$) a basis for the canonical extension $\cV_e$.  A basis of $\mathcal{F}^1_e\mathcal{V}_e$ is written $\omega_i=\ac_i+\sum_j\tau_{ij}\bc_j=\widetilde{\ac_i}+\sum_{j}\eta_{ij}\bc_j$, with $\eta$ holomorphic at $0$ and $\tau=\ft+\ay \ft^*$ symmetric and $\ft^*$ positive definite (on $\Delta^*$).

For a local (multivalued) lift of $R$ to $\cV$, we can write $\tilde{R}=\sum_{j=1}^g F_j\widetilde{\ac_j}+G_j\bc_j$.  To compute $\fR^*$, we need to rewrite $\tilde{R}$ with respect to the local system basis $\{\ac_j;\bc_j\}$, take imaginary parts of the coefficients, and finally rewrite again in the canonical extension basis, which yields single-valued coefficients.  However, the coefficient of $\widetilde{\ac_j}$ in $\fR^*$ is simply $\fF_j^*=\mathrm{Im}(F_j)$.
Deferring a detailed discussion of singularities to \S\ref{S4.1}, we keep things simple here:  the image of a generator of $E^{2,2}_{\lim}$ under $N$ is given by $\sum_{j=1}^g a_j\widetilde{\ac_j}(0)$ for some $a_j\in \QQ$; and $F_j=2\pi\ay a_j\log(z)+\{\text{holomorphic on }\Delta\}$, whence $\tfrac{1}{2\pi}\fF^*_j=a_j\log|z|+\{\text{harmonic on }\Delta\}$.  If any $a_j\neq 0$, after shrinking the disk the harmonic part is bounded by some $C$ which $|a_j\log|z||$ everywhere exceeds.  Hence $\fF_j^*$, \emph{a fortiori} $\fR^*$, has no zero in $\Delta^*$.  This reduces the problem to the case where $R$ is nonsingular at $0$.

In this case\footnote{N.B.: Here we do not break the argument into subcases $N=0$ and $N\neq 0$ as we did in \S\ref{S2}.} $R$ admits a single-valued lift
\begin{align*}
\tilde{R}\;&=\phantom{:}\textstyle\sum_i (F_i\ac_i+\hat{G}_i\bc_i)\;=\;\sum_iF_i\widetilde{\ac_i}+\sum_i(\hat{G}_i-\ell(z)\sum_j \mathcal{N}_{ij}F_j)\bc_i\\ 
&=:\textstyle\sum_i(F_i\widetilde{\ac_i}+G_i\bc_i)
\end{align*}
in which $F_i,G_i\in \mathcal{O}(\Delta)$ and $N\tilde{R}(0)=0$ (which means that for each $i$ we have $\sum_j \mathcal{N}_{ij}F_j(0)=0$). Differentiating gives
$$\textstyle\nabla_{\partial_z}\tilde{R}=\sum_{i=1}^g (f_i\ac_i+\hat{g}_i\bc_i),\;\;\;\;\;\hat{g}_i=\sum_j\tau_{ij}f_j=g_i+\ell(z)\sum_j\mathcal{N}_{ij}f_j,$$ for some $f_i$ and $g_i(=\sum_j\eta_{ij}f_j)$ which are holomorphic on $\Delta$.\footnote{Note that $F_i'=f_i$, $\hat{G}_i'=\hat{g}_i$, and $G_i'=g_i-\tfrac{1}{2\pi\ay z}\sum_j\mathcal{N}_{ij}F_j$, which is also holomorphic on $\Delta$.} Taking imaginary part of $\tilde{R}$ gives $$\fR^*=\textstyle\sum_i(\fF_i^*\ac_i+\hat{\fG}^*_i\bc_i)=\sum_i\fF_i^*\widetilde{\ac_i}+\sum_i \left(\hat{\fG}^*_i-\ell(z)\sum_j\mathcal{N}_{ij}\fF^*_j\right)\bc_i,$$
where $\hat{\fG}^*_i=\text{Im}(\hat{G}_i)=\text{Im}(G_i+\ell(z)\sum_j\mathcal{N}_{ij}F_j)=\fG^*_i-\tfrac{\log|z|}{2\pi}\sum\mathcal{N}_{ij}\fF_j+\tfrac{\arg(z)}{2\pi}\sum_j\mathcal{N}_{ij}\fF_j^*$.

\emph{Suppose there exists a sequence of points, limiting to $z=0$, on which all $\fF_i^*$ and $\hat{\fG}^*_i$ vanish} (or equivalently, that it is not possible to shrink $\Delta$ so that $Z_{\RR}(\hat{R}|_{\Delta^*})$ is finite). By Lemma \ref{L3.5} applied to $A_i=\fF^*_i$ ($p=g$, $n=2$, $W=\{0\}$), we find that their vanishing locus is either $\Delta$ or a real analytic variety of dimension 1.  In the first case, a holomorphic function with zero imaginary part is constant, which contradicts local nonconstancy of $\tilde{R}$.  In the second case, by \cite[Cor.~4.9]{BM} there is a sequence of (real) blowups at $0$ that makes the preimage normal-crossing.  This guarantees the existence of an injective analytic map $\gamma\colon I:=(-\epsilon,\epsilon)\to \Delta$ (sending $t\mapsto \gamma(t)=(x(t),y(t))$, and $0\mapsto (0,0)$) with image in the zero locus of the $\{\fF_i^*\}$.  Write $\gamma$ for both the image curve and the map, and $\gamma'(t)=\begin{psmatrix}x'(t)\\ y'(t) \end{psmatrix}$.

The derivative of the restriction of the imaginary part to this curve is now
$$\tfrac{d}{dt}\gamma^*\fR^*=\nabla_{\gamma'(t)}\fR^*(\gamma(t))=\begin{psmatrix}\fF^*_x& \fF^*_y\\ \hat{\fG}^*_x & \hat{\fG}^*_y\end{psmatrix}\gamma'(t)=\begin{psmatrix}\ff^* & \ff\\ \hat{\fg}^* & \hat{\fg}\end{psmatrix}\gamma'(t)$$
where the matrix entries are all evaluated at $\gamma(t)$, and the product is ``$2g \times 2$ matrix times $2\times 1$ vector''.  By construction, the first $g$ entries of the resulting $g\times 1$ vector (as analytic functions on $I$) are identically zero.  The last $g$ entries are derivatives of functions $\hat{\fG}^*_i(\gamma(t))$ (analytic on $I\setminus\{0\}$, continuous on $I$) with infinitely many zeroes approaching the origin.  So these derivatives must themselves have infinitely many zeroes, reflecting stationary points of $\hat{\fG}^*_i(\gamma(t))$.  So each entry of 
\begin{equation}\label{e-zeroes}
(\hat{\fg}^*\;\hat{\fg})\gamma'(t)=\ft(\ff^*\;\ff)\gamma'(t)+\ft^*(\ff\;-\ff^*)\gamma'(t)=\ft^*(\ff\;-\ff^*)\gamma'(t)
\end{equation}
has zeroes limiting to the origin.

At this point we cannot simply invert $\ft^*$, since the zeroes could occur in different entries at different sequences of values of $t$.  However,\footnote{Here $\fn^*=\text{Im}(\eta)$.} observe that $\ft^*(\gamma(t))=\fn^*(\gamma(t))-\tfrac{1}{2\pi}\log|z(t)|\mathcal{N}=\widetilde{\fn^*}(t)-c\tfrac{\log|t|}{2\pi}\mathcal{N}$, where $c=\tfrac{1}{2}\text{ord}_0(x(t)^2+y(t)^2)$.  Then \eqref{e-zeroes} is $$\widetilde{\fn^*}(\ff\;-\ff^*)\gamma'(t)-\log|t|\tfrac{c}{2\pi}\mathcal{N}(\ff\;-\ff^*)\gamma'(t)=:\mathbf{A}(t)-\log|t|\mathbf{B}(t)\,,$$ where $\mathbf{A}(t)$ and $\mathbf{B}(t)$ have analytic entries on $I$.  Applying Lemma \ref{L3.4} entry by entry, both are identically zero; hence so is \eqref{e-zeroes} and thus (\emph{now} inverting $\ft^*$) $(\ff\;-\ff^*)\gamma'(t)$.  Conclude that
$$(F_x\;F_y)\gamma'(t)=(\ff\;-\ff^*)\gamma'(t)+\ay(\ff^*\;\ff)\gamma'(t)\equiv 0,$$
making all $F_i(\gamma(t))$ constant.  But since the $F_i$ are holomorphic, they are then constant on the complex analytic closure of $\gamma$, which is $\Delta$.  This makes all $f_i$ zero, again contradicting the local nonconstancy of $\tilde{R}$.  So there was no sequence of points as described above, and we can shrink $\Delta$ to make $\fR^*$ nowhere vanishing on $\Delta^*$. This completes the proof of Claim \ref{C3.3}.

\section{The multiparameter setting}\label{S4}

In this section, we prove Theorem \ref{T2}. Let $\BV$ be a $\ZZ$-PVHS with Hodge numbers $(g,g)$ on a smooth complex quasi-projective variety $S$ of dimension $d\leq g$, and $R\in \ANF_S(\BV(2))$ a nontorsion normal function for which $\nabla R$ has maximal rank at every point of $S$.

\begin{lem}\label{L4.1}
Let $s\in S$ be a point and $\sD\subset S$ be a neighborhood of $s$.  Then after possibly shrinking $\sD$, $\fR^*|_{\sD}$ is either nowhere $0$ or zero only at $s$.  That is, zeroes of $\fR^*$ in $S$ are isolated.
\end{lem}
\begin{proof}
Regard $\fR^*=\mathrm{Im}(R)$ as a real analytic function from $\sD\subset\CC^d\cong \RR^{2d}$ to $\RR^{2g}$.  By Lemma \ref{L3.5} (with $W=\{0\}$), we can shrink $\sD$ so that $\{\fR^*=0\}\cap (\sD\sm \{0\})$ is either empty or contains a real analytic manifold of positive dimension. Assume the latter and place a coordinate neighborhood $\Delta^d$ at a point in this manifold.  (Make this point the new origin.)

By the maximal rank hypothesis, the $\mu_j(z):=(\nabla_{{\partial}/{\partial z_j}}R)(z)$ ($j=1,\ldots,d$) span a subspace of rank $d$ in $(\mathcal{F}^1\mathcal{V})_z$ for every $z\in \Delta^d$.  Writing $z_j=x_j+\ay y_j$, we have $\nabla_{{\partial}/{\partial x_j}}\fR^*=\text{Im}(\nabla_{{\partial}/{\partial x_j}}R)= \text{Im}(\nabla_{{\partial}/{\partial z_j}}R)=\text{Im}(\mu_j)$ and $\nabla_{{\partial}/{\partial y_j}}\fR^*=\text{Re}(\mu_j)$.

Since $\fR^*$ is zero on some manifold through $0$, for some nonzero tangent vector\footnote{Here $a$ and $b$ are both $d\times 1$ column vectors with real entries.} $\begin{psmatrix}a \\ b\end{psmatrix}\in \RR^{2d}\cong T_0\Delta^d$ we have $(\nabla_{\begin{psmatrix}a \\ b\end{psmatrix}}\fR^*)|_0=0$.  Then 
\begin{align*}
0&=\textstyle \sum_{j=1}^d(a_j \text{Im}(\mu_j)+b_j \text{Re}(\mu_j))(0)\\
&=\textstyle\sum_{j=1}^d\tfrac{1}{2}(b_j-\ay a_j)\mu_j(0)+\sum_{j=1}^d\tfrac{1}{2}(b_j+\ay a_j)\overline{\mu_j(0)}
\end{align*}
yields a contradiction as $\{\mu_j(0),\overline{\mu_j(0)}\}_{j=1}^d$ are independent over $\CC$.
\end{proof}

Now apply the covering Lemma \ref{L3.2} to replace $S$ by $\TS\subseteq \BS$ with $\Sigma=\BS\sm\TS$ a NCD, about which the basechange $\hat{\BV}$ has unipotent monodromies.  Clearly, since the cover is \'etale, $\nabla\hat{R}$ is everywhere of maximal rank and zeroes of $\hat{\fR}^*$ in $\TS$ are still isolated by Lemma \ref{L4.1}.  

Let $x\in \BS$ be an arbitrary point, and $\Delta^d\cong U\subset \BS$ a coordinate neighborhood of $x$ with $U\cap \TS\cong (\Delta^*)^{c}\times \Delta^{d-c}$ and coordinates $z_1,\ldots,z_d$.  We will show (taking $U$ sufficiently small) that $\hat{\fR}^*|_{U\cap \hat{S}}$ is (i) empty if $U\not\subset \TS$ and (ii) either empty or $\{x\}$ if $U\subset \TS$.  (Indeed, (ii) is already clear by the lemma.)  The compactness of $\BS$ then yields a finite cover by polydisks, in which (excluding the intersection of each polydisk with $\Sigma$, where $\hat{\fR}^*$ is not defined\footnote{Of course, we can extend $\hat{R}$ to those components of $\Sigma$ along which it is nonsingular.  It is entirely possible that zeroes of $\fR^*$ in this extension may not be isolated.  We are making no claim about this.}) the only possible zeroes of $\hat{\fR}^*$ are the centers of those polydisks which are contained in $\TS$.  This gives the desired finiteness of $Z_{\RR}(\widehat{R})$ hence that of $Z_{\RR}(R)$.

Before proceeding we spell out what is meant by (i) and simplify notation (dropping the hats).\footnote{We can also pass from $\ZZ$- to $\QQ$-VMHS, since we have already used the fact that $\BV$ is a $\ZZ$-PVHS to invoke the monodromy theorem.}  Suppose we have, for $1\leq c\leq d\leq g$,
\begin{itemize}[leftmargin=0.5cm]
\item a $\QQ$-PVHS $\BV$ on $U:=(\Delta^*)^{c}\times \Delta^{d-c}$ with Hodge numbers $(g,g)$ and (possibly trivial) unipotent monodromy operators $T_1,\ldots,T_{c}$; and
\item a nontorsion normal function $R\in \ANF_{U}(\BV(2))$ with the properties that zeroes of $\fR^*$ are isolated in $U$ and $\nabla R$ has rank $d$ everywhere.
\end{itemize}
In the subsections that follow we will show the
\begin{claim}\label{C4.2}
After shrinking $\overline{U}=\Delta^d$, $\fR^*$ is nowhere zero on $U$.
\end{claim}

As in \S\ref{S3.2}, we let $V$ denote a fiber of $\VV_{\QQ}$ on $U$, and $\{\ac_j,\bc_j\}_{j=1}^g$ denote a symplectic basis.  The latter may be chosen so that $N_1,\ldots,N_{c}$ take the form $N_k=\begin{psmatrix} 0 & 0 \\ -\cN^k & 0 \end{psmatrix}$ with $\cN^k=(\cN^k_{ij})$ symmetric; and then the $\bc_i$ and $\widetilde{\ac_i}:=\ac_i+\sum_{j=1}^g\sum_{k=1}^{c}\ell(z_k)\cN^k_{ij}\bc_j$ yield a basis for the unipotent extension.

\subsection{Nontrivial singularities}\label{S4.1}

In this part, where $0\neq \text{sing}_0(R)\in \Homm(\QQ(0),H^1(\imath_0^* R\jmath_*\BV(2)))$,\footnote{Here $\jmath\colon U\hookrightarrow \Delta^d$ and $\imath_0\colon \{0\}\hookrightarrow \Delta^d$ are the inclusions, and $R\jmath_* \BV(2)\in D^b MHM(\Delta^d)$.  We will spell out the singularity class explicitly (up to twist) as a tuple $(\fs_1,\ldots,\fs_c)$ below.} we make no use of the maximal rank assumption or its consequence (that zeroes of $\fR^*$ are isolated).  So the bounding of zeroes of $\fR^*$ away from a singularity of $R$ is a more general phenomenon.

Let $V_{\lim}$ denote the LMHS of $\BV$ at $0$, with monodromy weight filtration $M_{\bullet}:=W(N)[-1]_{\bullet}$, where $N=N_1+\cdots+N_{c}$.  We choose our symplectic basis so that it has graded pieces $\gr^M_2 V_{\lim}=\langle \acto(0),\ldots,\actr(0)\rangle$, $\gr^M_0 V_{\lim}=\langle \bc_1,\ldots,\bc_r\rangle$, and $\gr^M_1 V_{\lim}=\langle \ac_{r+1},\ldots,\ac_g;\bc_{r+1},\ldots,\bc_g\rangle$.  The monodromy logarithms $N_1,\ldots,N_{c}$ kill $M_1 V_{\lim}$ but (unlike $N$) may not individually induce isomorphisms from $\gr^M_2$ to $\gr^M_0$.  A result from \cite{CK} (which will be crucial for us) is that \emph{any $N'\in \RR_{>0}\langle N_1,\ldots,N_{c}\rangle$ induces the same filtration $M_{\bullet}$}.  Write $M^i_{\bullet}$ for the filtrations induced by the $N_i$.

Now consider the extension $0\to \BV(2)\to\BE\to \QQ(0)\to 0$ of $\QQ$-AVMHS on $U$ produced by $R$.  There are the monodromy logarithms $\tilde{N}_1,\ldots,\tilde{N}_{c}$ on $E_{\lim}$ lifting the ones on $V_{\lim}$, with sum inducing a weight filtration $\tilde{M}_{\bullet}$ on $E_{\lim}$ (restricting to $M[4]_{\bullet}$ on $V_{\lim}$). Denoting by $\mathbf{1}\in E_{\lim,\QQ}$ a lift of $1\in \QQ(0)\cong \gr^{\tilde{M}}_0E_{\lim,\QQ}$, we set $\fs_i:=\tilde{N}_i\mathbf{1}\in \gr^{\tilde{M}}_{-2}E_{\lim,\QQ}\cong \gr^{M}_2V_{\lim,\QQ}$.  We can represent $\fs_i=:\sum_{j=1}^r a^i_j\widetilde{\ac_j}(0)$ by a vector $a^i=(a^i_j)\in \QQ^r$.

Each of these $\fs_i$ is the limit of a singularity class along the $i^{\text{th}}$ coordinate hyperplane. Hence if $\fs_i\neq 0$, then $\fs_i$ is nonzero in $\gr^{M^i}_2V_{\lim}$, which $N_i$ sends isomorphically to $\gr^{M^i}_0 V_{\lim}$, and
\begin{equation}\label{prN1}
\fs_i\neq 0\implies N_i\fs_i\neq 0.
\end{equation}
Moreover, as $(\fs_1,\ldots,\fs_{c})$ gives a class in $H^1$ of the standard Koszul-type complex $V_{\lim}\overset{\oplus N_i}{\to} \oplus_{i=1}^{c} V_{\lim}(-1) \to \oplus_{i<j} V_{\lim}(-2)$, we see that
\begin{equation}\label{prN2}
N_i\fs_j=N_j\fs_i\;(\forall i,j).
\end{equation}
Finally, since the singularity class of $R$ at the origin is nonzero by assumption, some $\fs_i\neq 0$, wolog $\fs_1$.

For a local (multivalued) lift of $R$ to $\mathcal{V}$, we can write $\tilde{R}=\sum_{j=1}^g F_j\widetilde{\ac_j}+G_j\bc_j$, so that $\fF^*_j=\text{im}(F_j)$ is the coefficient of $\widetilde{\ac_j}$ in $\fR^*$.  The singularity class implies that $F_j=2\pi\ay\sum_{i=1}^{c} a^i_j\log(z_i)+\{\text{holomorphic on }\Delta^d\}$, whence $\fR_j^*:=\tfrac{1}{2\pi}\fF^*_j=\sum_{i=1}^{c}a^i_j\log|z_i|+\{\text{harmonic on }\Delta^d\}$.  If $c=1$, then at least one period blows up like $\log|z_1|$, which beats the (bounded) harmonic part on a possibly smaller polydisk.

So assume $2\leq c\leq d\leq g$, and suppose there is a sequence of points $z^{\gamma}=(z^{\gamma}_i)\to 0$ in $U$ with $\fR^*(z^{\gamma})=0$ (\emph{a fortiori} all $\fR^*_j(z^{\gamma})=0$).  Since the harmonic parts are bounded below $\tfrac{1}{\sqrt{g}}$ on a polydisk, for each $K\in \mathbb{N}$ there must exist $L_i(K)\in \RR_{>K}$ ($i=1,\ldots,c$) such that $\|\sum_i L_i(K)a^i\|<1$.  (This is simply by taking $\gamma$ sufficiently large that $-\log|z^{\gamma}_i|>K$.) Since $\PP^{c-1}(\RR_{\geq 0})$ is compact, there is a subsequence $K_p$ such that $[L_i(K_p)]$ has a limit $[\cL_i]\in \PP^{c-1}(\RR_{\geq 0})$.  We may assume that $\cL_1\neq 0$, so that $\|a^1+\sum_{i>1}\tfrac{L_i(K_p)}{L_1(K_p)}a^i\|<\tfrac{1}{L_1(K_p)}<\tfrac{1}{K_p}\to 0$ $\implies$ $$\textstyle\|a^1+\sum_{i>1}\tfrac{\cL_i}{\cL_1}a^i\|\leq \|\sum_{i>1}\left(\tfrac{\cL_i}{\cL_1}-\tfrac{L_i(K_p)}{L_1(K_p)}\right)a^i\|+\|a^1+\sum_{i>1}\tfrac{L_i(K_p)}{L_1(K_p)}a^i\|\to 0$$
$\implies$ $\|a^1+\sum_{i>1}\tfrac{\cL_i}{\cL_1}a^i\|=0$.  So $a^1=-\sum_{i=2}^{c}\cL_i a^i=0$, or equivalently, 
\begin{equation}\label{prN3}
\textstyle\fs_1=-\sum_{i\geq 2}\tfrac{\cL_i}{\cL_1}\fs_i.
\end{equation}

Now by \eqref{prN3} and \eqref{prN2} we have 
\begin{align*}
\textstyle(\sum_{i=1}^{c}\cL_i N_i)\fs_1&=\textstyle\cL_1N_1\fs_1+\sum_{i\geq 2}\cL_i N_i\fs_1\\
&=-\textstyle\sum_{i\geq 2}\cL_iN_1\fs_i+\sum_{i\geq 2}\cL_i N_i\fs_1=0.	
\end{align*}
But combining this with \eqref{prN1} gives $(N_1+\sum_{i=1}^{c}\cL_iN_i)\fs_1=N_1\fs_1\neq 0$.  Hence $N'=\sum_{i=1}^{c}\cL_iN_i$ and $N''=N_1+\sum_{i=1}^{c}\cL_iN_i$, which are both in $\RR_{>0}\langle N_1,\ldots,N_{c}\rangle$, induce different weight filtrations:  for the first, $\fs_1\in M'_1$ (as $N'$ kills it); for the second, $\fs_1\notin M''_1$ (as $N''$ doesn't).  This contradicts \cite[Thm.~3.3]{CK}, proving Claim \ref{C4.2} in this case.

\subsection{Trivial singularities}\label{S4.2}

To make the idea clear, we explain first the case where the $T_k$ are trivial ($N_k=0$), so that \emph{both} $\BV$ and $R$ are nonsingular at $0$.  Write $\tilde{R}=\sum_{i=1}^g\{F_i\ac_i+G_i\bc_i\}$, with $F_i=\fF_i+\ay\fF_i^*$ and $G_i=\fG_i+\ay\fG_i^*$ for the decomposition into real and imaginary parts.  Since $R$ is nonsingular these extend to holomorphic functions on $\Delta^d$, with harmonic (hence real-analytic) real and imaginary parts.  We apply Lemma \ref{L3.5} with $n=2d$, $W=\Delta^d\sm U$, and $X$ the vanishing locus of $\fR^*$, i.e.~of the harmonic functions $\{\fF_1^*,\ldots,\fF_g^*;\fG_1^*,\ldots,\fG_g^*\}$ of $\{x_1,\ldots,x_d;y_1,\ldots,y_d\}$ (real and imaginary parts of $z_1,\ldots,z_d$).

Suppose that we are in alternative (ii) of the Lemma applied to this setting.  In particular, these $2g$ functions are constant along a positive-dimensional real-analytic manifold $Y\subset U$.  Write 
\begin{itemize}[leftmargin=0.5cm]
\item $\fF^*_x:=(\tfrac{\partial\fF^*_i}{\partial x_j})$ ($g\times d$ matrix) and so forth (for $x$ and $y$; $F,\fF^*,\fF,G,\fG^*,$ and $\fG$), so that the Cauchy-Riemann equations give e.g.~$\fF_x^*=-\fF_y$ and $\fF^*_y=\fF_x$;
\item $\omega_i=\ac_i+\sum_j \tau_{ij}\bc_j$ ($i=1,\ldots,g$) for the standard basis of $\mathcal{F}^1\mathcal{V}$, with $\tau=\ft+\ay\ft^*$ the real/imaginary decomposition, and $\tau$ symmetric with $\ft^*$ positive definite on $\Delta^d$; and 
\item $\nabla_{\partial_{z_k}}\tilde{R}=\sum_i f_{ik}\omega_i$ ($k=1,\ldots,d$; $f=\ff+\ay\ff^*$) so that in matrix notation $DF:=(\tfrac{\partial F_i}{\partial z_k})=f$ and $DG=\tau f$.  (Also note that $DF=F_x=-\ay F_y$, so $\ff=\fF_x$ etc.)
\end{itemize}
Then for any tangent vector $\begin{psmatrix}a\\ b\end{psmatrix}\in \RR^{2d}$ at a point $z\in Y$, we have $$\begin{psmatrix}\ff^* & \ff \\ \text{Im}(\tau f) & \text{Re}(\tau f) \end{psmatrix}\begin{psmatrix}a\\ b\end{psmatrix}= \begin{psmatrix}\fF_x^* & \fF_x\\ \fG_x^* & \fG_x\end{psmatrix}\begin{psmatrix}a\\ b\end{psmatrix}=\begin{psmatrix}\fF_x^* & \fF_y^*\\ \fG_x^* & \fG_y^*\end{psmatrix}\begin{psmatrix}a\\ b\end{psmatrix}=0,$$
hence $(\ff^*\;\ff)\begin{psmatrix}a\\ b\end{psmatrix}=0$ and $0=\ft(\ff^*\;\ff)\begin{psmatrix}a\\ b\end{psmatrix}+\ft^*(\ff\;-\ff^*)\begin{psmatrix}a\\ b\end{psmatrix}=\ft^*(\ff\;-\ff^*)\begin{psmatrix}a\\ b\end{psmatrix}$.  As $\ft^*$ is invertible on $\Delta^d$, this gives $(\ff\;-\ff^*)\begin{psmatrix}a\\ b\end{psmatrix}=0$.  Since we have $(\ff\;-\ff^*)+\ay(\ff^*\;\ff)=(F_x\;F_y)$, we get that $D_{\begin{psmatrix}a\\ b\end{psmatrix}}F=0$, and conclude that $F$ is constant along $Y$.  But then since $F$ is complex analytic, it is constant along the complex analytic closure $\cY$ of $Y$ in $U$.  Since $DG=\tau DF$, $\tilde{R}$ is constant on $\cY$ and so $\mathrm{rk}(\nabla \tilde{R})<d$ along $\cY$, a contradiction.  So we are in alternative (i) of the Lemma \ref{L3.5}, which yields Claim \ref{C4.2} in this subcase.

However, we would still be a long way from being able to deal with nontrivial $\{T_k\}$ without one additional tool.  This is the place in the proof where an input from o-minimal theory becomes critical.  By a \emph{log-real-analytic} function on $\Delta^d$, we shall mean a polynomial in the $\log|z_i|=\tfrac{1}{2}\log(x_i^2+y_i^2)$ with coefficients in real analytic functions on $\Delta^d$.

\begin{lem}\label{L4.3}
Let $A_1,\ldots,A_p$ be log-real-analytic functions on $\Delta^d$, with common vanishing locus $X$.  Then either (i) it is possible to shrink $\Delta^d$ so that $X^-:=X\cap U$ becomes empty, or (ii) $X^-$ contains a connected real-analytic curve.
\end{lem}
\begin{proof}
Let $\mathscr{C}$ be the analytic-geometric category corresponding to the o-minimal structure $\RR_{\text{an,exp}}$ \cite[\S2.5(6)]{vdDM}.  Then $X^-$ belongs to $\mathscr{C}(\Delta^d)$ (not just $\mathscr{C}(U)$).  If (i) doesn't hold, there is a sequence of points in $X^-$ limiting to $0\in \Delta^d$, and so $0$ belongs to the \emph{frontier} of $X^-$.  By [op.~cit., \S1.17], there exists a continuous injective map $\gamma\colon [0,1)\to \Delta^d$ such that $\gamma(0,1)\subseteq X^-$, $\gamma|_{(0,1)}$ is real analytic, and $\gamma(0)=0$.
\end{proof}

To eliminate case (ii) and verify Claim \ref{C4.2}, we now modify the above proof as follows.  First, we may write
\begin{align*}
\tilde{R}&=\textstyle\sum_{i=1}^g (F_i\ac_i+\hat{G}_i\bc_i)\\
&=\textstyle\sum_{i=1}^g F_i\widetilde{\ac_i}+\sum_{i=1}^g (\hat{G}_i-\sum_{j=1}^g\sum_{k=1}^{c}\ell(z_k)\cN_{ij}^kF_j)\bc_i\\
&=:\textstyle \sum_{i=1}^g(F_i\widetilde{\ac_i}+G_i\bc_i)
\end{align*}
with $F_i,G_i\in\cO(\Delta^d)$.  Taking imaginary parts gives
\begin{align*}
\fR^*&=\textstyle\sum_{i=1}^g (\fF^*_i\ac_i+\hat{\fG}_i^*\bc_i)=\textstyle\sum_i\fF^*_i\widetilde{\ac_i}+\sum_i(\hat{\fG}^*_i-\sum_{j,k}\ell(z_k)\cN_{ij}^k\fF^*_j)\bc_i\\
&=\textstyle\sum_i\fF_i^*\widetilde{\ac_i}+\sum_i(\fG^*_i-\tfrac{1}{2\pi}\sum_{j,k}\log|z_k|\cN^k_{ij}(\fF_j-\ay\fF^*_j))\bc_i.
\end{align*}
If there is a sequence of points in $U$ on which $\fR^*$ vanishes, then they belong to the common zero-locus of the (resp.~real-analytic and log-real-analytic) functions $\fF^*_i$ and $\widetilde{\fG_i^*}:=\fG^*_i-\tfrac{1}{2\pi}\sum_{j,k}\log|z_k|\cN^k_{ij}\fF_j$ ($i=1,\ldots,g$).  By Lemma \ref{L4.3}, these functions must vanish along a real-analytic curve $\gamma$ in $U$.  Clearly $\fR^*$ does as well.

Writing $DF=f$ and $D\hat{G}=\hat{g}=\tau f$, we then have
$$0=\tfrac{d}{dt}\gamma^*\fR^*=\begin{psmatrix}\fF_x^* & \fF_y^* \\ \hat{\fG}^*_x & \hat{\fG}^*_y \end{psmatrix}\gamma'(t) =\begin{psmatrix}\ff^* & \ff \\ \hat{\fg}^* & \hat{\fg} \end{psmatrix}\gamma'(t) =\begin{psmatrix}\ff^* & \ff \\ \text{Im}(\tau f) & \text{Re}(\tau f) \end{psmatrix}\gamma'(t)\,.
$$
At this point the proof is as before, using invertibility of $\ft^*$ along $\gamma$ to show that $F$ is constant on the complex analytic closure of $\gamma$ in $U$, contradicting the maximal rank assumption there.  So there cannot have been such a sequence of points, and Claim \ref{C4.2} holds in this case too.

\section{Application to Laurent polynomials}\label{S5}

It remains to provide a geometric scenario which produces normal functions as in Theorem \ref{T2}.  As in the Introduction, let $\Delta\subset \RR^2$ be a polygon with integer vertices and $g$ interior integer points; and let $F=F_a$ be the family of tempered Laurent polynomials in \eqref{eq1.2}, depending on $a\in S:=\bA^g\sm \Sigma$.  This produces a family of smooth complete curves $\PP_{\Delta}\times S\supset \cC\overset{\pi}{\to}S$ and a higher Chow cycle $\xi\in \mathrm{CH}^2(\cC,2)$ whose restriction to $\cC^*:=\cC\cap (\GG_m^2\times S)$ is the coordinate symbol $\{-x_1,-x_2\}$ (or a nonzero integer multiple thereof).

The images of $\xi_a\in \mathrm{CH}^2(\cC_a,2)$ under \eqref{eq1.3} give rise to an admissible normal function $R\in \mathrm{ANF}_S(\BV(2))$, where $\BV=R^1\pi_*\ZZ$.  We will only need the following, see \cite[\S\S3-4]{DK} and \cite[App.~A]{DKS} for more details.  The restriction of $R_a\in H^1(C_a,\CC/\ZZ(2))$ to $R_a^*\in H^1(C_a^*,\CC/\ZZ(2))$ is represented by the restriction to $X_a^*$ of a global current $R\{-x_1,-x_2\}$.  The exterior derivative of this current is $\tfrac{dx_1}{x_1}\wedge\tfrac{dx_2}{x_2}$, up to a $\ZZ(2)$-valued current which we can bound locally over the base.  So we have $\nabla R^*=\tfrac{dx_1}{x_1}\wedge \tfrac{dx_2}{x_2}$ on $\cC^*$.

Now the residue of $\tfrac{dx_1}{x_1}\wedge\tfrac{dx_2}{x_2}$ on a component of $\cC \setminus\cC^*$ is $\mathrm{dlog}$ of the corresponding root of an edge polynomial of $F$; as this is a root of unity hence constant, the residue is zero.  Thus $\tfrac{dx_1}{x_1}\wedge\tfrac{dx_2}{x_2}$ extends to a holomorphic form $\Omega\in \Omega^2(\cC)$.  Moreover, $\nabla R$ is represented by a holomorphic 2-form on $\cC$ by Griffiths transversality.  Since $\Omega^2(\cC)\hookrightarrow \Omega^2(\cC^*)$, we must therefore have $\nabla R=\Omega$.  Since $\cC^*$ is cut out of $\GG_m^2\times S$ by $F$, we can also write $\Omega$ as $\mathrm{Res}_{\cC}(\tfrac{dF}{F}\wedge \tfrac{dx_1}{x_1}\wedge \tfrac{dx_2}{x_2})$ on $\cC^*$ hence on $\cC$.  For the ``directional'' derivatives we therefore have
\begin{align*}
(\nabla_{\partial_{a_j}}R)(a)&=\langle \nabla R_a,\partial_{a_j}\rangle|_{C_a} = \langle \mathrm{Res}_{\mathcal{C}}(\tfrac{dF}{F}\wedge\tfrac{dx_1}{x_1}\wedge \tfrac{dx_2}{x_2}),\widetilde{\partial_{a_j}}\rangle|_{C_a}\\
&=\mathrm{Res}_{C_a}(\langle \tfrac{dF}{F}\wedge\tfrac{dx_1}{x_1}\wedge\tfrac{dx_2}{x_2},\partial_{a_j}\rangle)\\
&=\mathrm{Res}_{C_a}(\tfrac{\partial_{a_j}F}{F}\wedge \tfrac{dx_1}{x_1}\wedge\tfrac{dx_2}{x_2})\\
&=\mathrm{Res}_{C_a}\left(\frac{x_1^{m_1^j}x_2^{m^j_2}\frac{dx_1}{x_1}\wedge \frac{dx_2}{x_2}}{F_a(x_1,x_2)}\right)=:\mu_j(a).
\end{align*}
By \cite[p.~386]{Ba}, for \emph{every} $a\in S$ these $\mu_1,\ldots,\mu_g$ yield a basis of $\Omega^1(C_a)$.  In particular, they are independent, and $\nabla R$ has maximal rank.  This proves Proposition \ref{P3}.

\appendix
\section{Some matrix differential equations}\label{AA}

In this appendix we present an alternate proof of Claim \ref{C4.2} (without using \cite{BM} or \cite{vdDM}) in the case where both $R$ and $\BV$ have trivial singularities.  The key here is Lemma \ref{LA2}, which solves an interesting problem in PDE for which we could not find a reference.  We thank J.~McCarthy for discussions on this point.

\begin{thm}\label{TA1}
Let $B \subset \complex^d$ be an open ball centered at the origin, $F, G \colon B \to \complex^g$ vectors of holomorphic functions vanishing at the origin, and $\tau \colon B \to M_{g \times g}(\complex)$ a matrix of holomorphic functions satisfying $dG = \tau dF$\footnote{We think of $dF$ and $dG$ as vectors of one-forms.} and $\det(\im{t}(0)) \neq 0$. If $V(\im{F}, \im{G})$ contains a sequence of points $\{x_i \neq 0\}$ converging to the origin, then a subsequence lies on a positive-dimensional component of $V(F, G)$ passing through the origin.
\end{thm}

\begin{proof}
    We apply complex analytic embedded resolution of singularities to obtain a proper surjective holomorphic map $\pi \colon (\tilde{B}, \tilde{D}) \to (B, D)$, where $D$ is the zero divisor of the product $F_1 F_2 \cdots F_g$. The properties of $\pi$ yield a sequence $\{\tilde{x}_i\}$ in the preimage of $\{x_i\}$ converging to a point $p \in \pi^{-1}(0)$. There is a coordinate ball $U$ centered at $p$ in which the components of $\tilde{F} \coloneqq F \circ \pi$ take the form $\tilde{F}_j = u_j z^{a_j}$, with $u_j$ a unit in the analytic local ring $\mathcal{O}_{U, 0} \cong \complex\{z_1, z_2, \dots, z_g\}$ and $a_j$ a multi-index. We will henceforth work in $U$ and drop the tildes. We construct a matrix $S$ as in the lemma below. Shrinking $U$ so that $\im{S}$ is everywhere invertible, we observe
    \[
    0 = \im{G}(x_i) = \im{S}(x_i)\re{F}(x_i) + \re{S}(x_i)\im{F}(x_i) = \im{S}(x_i)\re{F}(x_i),
    \]
    which implies $\re{F}(x_i) = 0$ by invertibility of $\mathfrak{S}^*(x_i)$ and hence $G(x_i) = F(x_i) = 0$. The image of these points is a sequence downstairs converging to the origin and contained in $V(F, G)$, which is a complex analytic variety.
\end{proof}

\begin{lem}\label{LA2}
Let $F, G, \tau$ be defined on $U$ as above, with $V(F_1\cdots F_g)$ a union of coordinate hyperplanes. Then there is a matrix of holomorphic functions $S \colon U \to M_{g \times g}(\complex)$ satisfying $G = SF$ and $S(0) = \tau(0)$.
\end{lem}

\begin{proof}
    Define $S \coloneqq \tau - \varepsilon$, where
\begin{equation}\label{eA}
    \varepsilon_{ij}(z) \coloneqq \frac{1}{F_j(z)} \int_{\gamma_z} F_j d\tau_{ij},
\end{equation}
    and $\gamma_z \colon [0, 1] \to U$ is the line segment $t \mapsto tz$.  (If $F_j$ vanishes identically, then we set $\varepsilon_{ij}\coloneqq 0$ for all $i$.)  A routine computation verifies that $G = SF$ where $S$ is defined.  We claim that after possibly shrinking the radius, $S$ is defined on all of $B$.
 
The restriction of $\varepsilon_{ij}$ to the $z_1$-axis is a one-variable holomorphic function vanishing at the origin.  Therefore, $\varepsilon_{ij}$ is either zero or indeterminate at the origin.  Now by hypothesis, each $F_j$ is associate to a monomial, with zero divisor is a union of coordinate hyperplanes. This forces the integral \eqref{eA} to vanish whenever $F_j$ does, meaning that every $\varepsilon_{ij}$ is holomorphic after possibly shrinking $B$.
Consequently, $\varepsilon(0) = 0$.  (We may also proceed by a power series argument. Working in the analytic local ring and ignoring units, let $F_j = z^a$ and $\frac{\partial \tau_{ij}}{\partial z_k} = \sum_{b \geq 0} c_b z^b$. Then
    \[
    \frac{1}{F_j(z)}\int_{\gamma_z} F_j \frac{\partial \tau_{ij}}{\partial z_k} dz_k = z_k \sum_{b \geq 0} \frac{c_b}{|a| + |b| + 1} z^b,
    \]
    showing that $\varepsilon(0) = 0$.)  This completes the proof.
\end{proof}

\begin{rem}
In the proof of Lemma \ref{LA2}, we at least need the divisors of the $\{F_j\}$ to be star-shaped with respect to the origin.  Otherwise indeterminacy cannot be avoided:  for instance, in
    \begin{align*}
        F =
        \begin{bmatrix}
        z_1 + z_2^2\\
        z_2^2
        \end{bmatrix}
        &&
        \tau =
        \begin{bmatrix}
            z_1 + \mathbf{i} & -z\\
            -z & z_1 + \mathbf{i}
        \end{bmatrix} &&
        \varepsilon_{11} = \frac{z_1^2/2 + z_1 z_2^2/3}{z_1 + z_2^2},
    \end{align*}
$\varepsilon_{11}$ is indeterminate at $0$.
\end{rem}

Finally, if we remove the assumption that $F$ and $G$ vanish at $0$, then applying Theorem \ref{TA1} to $F-\mathrm{Re}(F(0))$ and $G-\mathrm{Re}(G(0))$ yields a positive-dimensional complex-analytic variety $Y$ through the origin on which $F$ and $G$ are constant.  (If the sequence $\{x_i\}$ avoids some $\Sigma\subset B$, then $Y$ is obviously not contained in $\Sigma$.) This contradicts the maximal rank assumption on $\nabla R$ and yields the alternative proof claimed at the start of this appendix.

\bibliographystyle{alpha}
\bibliography{AAKrefs.bib}

@Article{BBT,
 Author = {Bakker, Benjamin and Brunebarbe, Yohan and Tsimerman, Jacob},
 Title = {o-minimal {GAGA} and a conjecture of {Griffiths}},
 FJournal = {Inventiones Mathematicae},
 Journal = {Invent. Math.},
 ISSN = {0020-9910},
 Volume = {232},
 Number = {1},
 Pages = {163--228},
 Year = {2023},
 Language = {English},
 DOI = {10.1007/s00222-022-01166-1},
 zbMATH = {7662555},
 Zbl = {1516.14024}
}

@misc{BU,
      title={Effective atypical intersections and applications to orbit closures}, 
      author={Gregorio Baldi and David Urbanik},
      year={2024},
      eprint={2406.16628},
      archivePrefix={arXiv},
      primaryClass={math.AG},
      url={https://arxiv.org/abs/2406.16628}, 
      note={arXiv:2406.16628}
}

@misc{KT,
      title={On the torsion locus of the {Ceresa} normal function}, 
      author={Matt Kerr and Salim Tayou},
      year={2024},
      eprint={2406.19366},
      archivePrefix={arXiv},
      primaryClass={math.AG},
      url={https://arxiv.org/abs/2406.19366},
      note={arXiv:2406.19366},
}

@misc{KlT,
	title={The {Zilber}-{Pink} conjecture for varieties not defined over $\overline{\QQ}$},
	author={Bruno Klingler and Salim Tayou},
	year={2025},
	eprint={2504.00865},
	archivePrefix={arXiv},
	primaryClass={math.AG},
	url={https://arxiv.org/abs/2504.00865},
	note={arXiv:2504.00865},
}

@misc{GZ,
      title={Heights and periods of algebraic cycles in families}, 
      author={Ziyang Gao and Shou-Wu Zhang},
      year={2024},
      eprint={2407.01304},
      archivePrefix={arXiv},
      primaryClass={math.AG},
      url={https://arxiv.org/abs/2407.01304},
      note={arXiv:2407.01304}
}

@article{GKS,
author = {Golyshev, Vasily and Kerr, Matt and Sasaki, Tokio},
title = {Apéry extensions},
journal = {Journal of the London Mathematical Society},
volume = {109},
number = {1},
pages = {e12825},
doi = {https://doi.org/10.1112/jlms.12825},
url = {https://londmathsoc.onlinelibrary.wiley.com/doi/abs/10.1112/jlms.12825},
eprint = {https://londmathsoc.onlinelibrary.wiley.com/doi/pdf/10.1112/jlms.12825},
year = {2024}
}

@Article{BP,
 Author = {Brosnan, Patrick and Pearlstein, Gregory},
 Title = {On the algebraicity of the zero locus of an admissible normal function},
 FJournal = {Compositio Mathematica},
 Journal = {Compos. Math.},
 ISSN = {0010-437X},
 Volume = {149},
 Number = {11},
 Pages = {1913--1962},
 Year = {2013},
 Language = {English},
 DOI = {10.1112/S0010437X1300729X},
 zbMATH = {6255403},
 Zbl = {1293.32019}
}

@ARTICLE{SZ,
    AUTHOR = {Steenbrink, Joseph and Zucker, Steven},
     TITLE = {Variation of mixed {H}odge structure. {I}},
   JOURNAL = {Invent. Math.},
  FJOURNAL = {Inventiones Mathematicae},
    VOLUME = {80},
      YEAR = {1985},
    NUMBER = {3},
     PAGES = {489--542},
      ISSN = {0020-9910,1432-1297},
   MRCLASS = {32G20 (14C30 32J25)},
  MRNUMBER = {791673},
MRREVIEWER = {Sampei\ Usui},
       DOI = {10.1007/BF01388729},
       URL = {https://doi.org/10.1007/BF01388729},
}

@incollection{FRV,
    AUTHOR = {Villegas, F. Rodriguez},
     TITLE = {Modular {M}ahler measures. {I}},
 BOOKTITLE = {Topics in number theory ({U}niversity {P}ark, {PA}, 1997)},
    SERIES = {Math. Appl.},
    VOLUME = {467},
     PAGES = {17--48},
 PUBLISHER = {Kluwer Acad. Publ., Dordrecht},
      YEAR = {1999},
      ISBN = {0-7923-5583-0},
   MRCLASS = {11G40 (11R06 19F27)},
  MRNUMBER = {1691309},
MRREVIEWER = {Jan\ Nekov\'a\v r},
}

@incollection{KP,
    AUTHOR = {Kerr, Matt and Pearlstein, Gregory},
     TITLE = {An exponential history of functions with logarithmic growth},
 BOOKTITLE = {Topology of stratified spaces},
    SERIES = {Math. Sci. Res. Inst. Publ.},
    VOLUME = {58},
     PAGES = {281--374},
 PUBLISHER = {Cambridge Univ. Press, Cambridge},
      YEAR = {2011},
      ISBN = {978-0-521-19167-8},
   MRCLASS = {14D07 (14K30 32G20 32S40)},
  MRNUMBER = {2796415},
MRREVIEWER = {Christian\ Schnell},
}

@article {Ke,
    AUTHOR = {Kerr, Matt},
     TITLE = {Unipotent extensions and differential equations (after
              {B}loch-{V}lasenko)},
   JOURNAL = {Commun. Number Theory Phys.},
  FJOURNAL = {Communications in Number Theory and Physics},
    VOLUME = {16},
      YEAR = {2022},
    NUMBER = {4},
     PAGES = {801--849},
      ISSN = {1931-4523,1931-4531},
   MRCLASS = {14D07 (14C30 14F42 19E15 32G20 32S40)},
  MRNUMBER = {4503993},
MRREVIEWER = {Noriko\ Yui},
}

@article{vdDM,
    AUTHOR = {van den Dries, Lou and Miller, Chris},
     TITLE = {Geometric categories and o-minimal structures},
   JOURNAL = {Duke Math. J.},
  FJOURNAL = {Duke Mathematical Journal},
    VOLUME = {84},
      YEAR = {1996},
    NUMBER = {2},
     PAGES = {497--540},
      ISSN = {0012-7094,1547-7398},
   MRCLASS = {32B20 (03C68)},
  MRNUMBER = {1404337},
MRREVIEWER = {Anand\ Pillay},
       DOI = {10.1215/S0012-7094-96-08416-1},
       URL = {https://doi.org/10.1215/S0012-7094-96-08416-1},
}

@article {CK,
    AUTHOR = {Cattani, Eduardo and Kaplan, Aroldo},
     TITLE = {Polarized mixed {H}odge structures and the local monodromy of
              a variation of {H}odge structure},
   JOURNAL = {Invent. Math.},
  FJOURNAL = {Inventiones Mathematicae},
    VOLUME = {67},
      YEAR = {1982},
    NUMBER = {1},
     PAGES = {101--115},
      ISSN = {0020-9910,1432-1297},
   MRCLASS = {32J25 (14C30 22E99 32M10)},
  MRNUMBER = {664326},
MRREVIEWER = {Steven\ M.\ Zucker},
       DOI = {10.1007/BF01393374},
       URL = {https://doi.org/10.1007/BF01393374},
}

@misc{Kl,
      title={Hodge loci and atypical intersections: conjectures}, 
      author={Bruno Klingler},
      year={2017},
      eprint={1711.09387},
      archivePrefix={arXiv},
      primaryClass={math.AG},
      url={https://arxiv.org/abs/1711.09387},
      note={arXiv:1711.09387, to appear in ``Motives and Complex Multiplication.''},
}

@misc{Ak,
	title={The {A}-factor locus of a normal function},
	author={Devin Akman},
	year={2025},
	note={to appear in ``Regulators V'' proceedings.},
}

@article{BKU,
    AUTHOR = {Baldi, Gregorio and Klingler, Bruno and Ullmo, Emmanuel},
     TITLE = {On the distribution of the {H}odge locus},
   JOURNAL = {Invent. Math.},
  FJOURNAL = {Inventiones Mathematicae},
    VOLUME = {235},
      YEAR = {2024},
    NUMBER = {2},
     PAGES = {441--487},
      ISSN = {0020-9910,1432-1297},
   MRCLASS = {14D07 (03C64 14C30 14G35 22F30)},
  MRNUMBER = {4689371},
       DOI = {10.1007/s00222-023-01226-0},
       URL = {https://doi.org/10.1007/s00222-023-01226-0},
}

@article{BM,
    AUTHOR = {Bierstone, Edward and Milman, Pierre D.},
     TITLE = {Semianalytic and subanalytic sets},
   JOURNAL = {Inst. Hautes \'Etudes Sci. Publ. Math.},
  FJOURNAL = {Institut des Hautes \'Etudes Scientifiques. Publications
              Math\'ematiques},
    NUMBER = {67},
      YEAR = {1988},
     PAGES = {5--42},
      ISSN = {0073-8301,1618-1913},
   MRCLASS = {32B20 (58C06 58C27)},
  MRNUMBER = {972342},
MRREVIEWER = {Z.\ Denkowska},
       URL = {http://www.numdam.org/item?id=PMIHES_1988__67__5_0},
}

@article {GM,
    AUTHOR = {Guilloux, Antonin and March\'e, Julien},
     TITLE = {Volume function and {M}ahler measure of exact polynomials},
   JOURNAL = {Compos. Math.},
  FJOURNAL = {Compositio Mathematica},
    VOLUME = {157},
      YEAR = {2021},
    NUMBER = {4},
     PAGES = {809--834},
      ISSN = {0010-437X,1570-5846},
   MRCLASS = {11R06 (14T20 19C99 57K14)},
  MRNUMBER = {4247573},
MRREVIEWER = {Detchat\ Samart},
       DOI = {10.1112/s0010437x21007016},
       URL = {https://doi.org/10.1112/s0010437x21007016},
}

@article{CCGLS,
    AUTHOR = {Cooper, D. and Culler, M. and Gillet, H. and Long, D. D. and
              Shalen, P. B.},
     TITLE = {Plane curves associated to character varieties of
              {$3$}-manifolds},
   JOURNAL = {Invent. Math.},
  FJOURNAL = {Inventiones Mathematicae},
    VOLUME = {118},
      YEAR = {1994},
    NUMBER = {1},
     PAGES = {47--84},
      ISSN = {0020-9910,1432-1297},
   MRCLASS = {57N10 (57M25)},
  MRNUMBER = {1288467},
MRREVIEWER = {Serge\ L.\ Tabachnikov},
       DOI = {10.1007/BF01231526},
       URL = {https://doi.org/10.1007/BF01231526},
}

@article {DKS,
    AUTHOR = {Doran, Charles F. and Kerr, Matt and Sinha Babu, Soumya},
     TITLE = {{$K_2$} and quantum curves},
   JOURNAL = {Adv. Theor. Math. Phys.},
  FJOURNAL = {Advances in Theoretical and Mathematical Physics},
    VOLUME = {27},
      YEAR = {2023},
    NUMBER = {8},
     PAGES = {2261--2318},
      ISSN = {1095-0761,1095-0753},
   MRCLASS = {14J33 (14F42 81T30)},
  MRNUMBER = {4788090},
       DOI = {10.4310/atmp.2023.v27.n8.a1},
       URL = {https://doi-org.libproxy.washu.edu/10.4310/atmp.2023.v27.n8.a1},
}

@article {DK,
    AUTHOR = {Doran, Charles F. and Kerr, Matt},
     TITLE = {Algebraic {$K$}-theory of toric hypersurfaces},
   JOURNAL = {Commun. Number Theory Phys.},
  FJOURNAL = {Communications in Number Theory and Physics},
    VOLUME = {5},
      YEAR = {2011},
    NUMBER = {2},
     PAGES = {397--600},
      ISSN = {1931-4523,1931-4531},
   MRCLASS = {19E15 (11G42 14C30 14J32 14J33 14M25 19E08)},
  MRNUMBER = {2851155},
MRREVIEWER = {Jan\ Nagel},
       DOI = {10.4310/CNTP.2011.v5.n2.a3},
       URL = {https://doi.org/10.4310/CNTP.2011.v5.n2.a3},
}

@book {Bo,
    AUTHOR = {Borel, Armand},
     TITLE = {Introduction aux groupes arithm\'etiques},
    SERIES = {Publications de l'Institut de Math\'ematique de
              l'Universit\'e{} de Strasbourg, XV. Actualit\'es Scientifiques
              et Industrielles},
    VOLUME = {No. 1341},
 PUBLISHER = {Hermann, Paris},
      YEAR = {1969},
     PAGES = {125},
   MRCLASS = {14.50 (20.00)},
  MRNUMBER = {244260},
MRREVIEWER = {James\ E.\ Humphreys},
}

@article {Ba,
    AUTHOR = {Batyrev, Victor V.},
     TITLE = {Variations of the mixed {H}odge structure of affine
              hypersurfaces in algebraic tori},
   JOURNAL = {Duke Math. J.},
  FJOURNAL = {Duke Mathematical Journal},
    VOLUME = {69},
      YEAR = {1993},
    NUMBER = {2},
     PAGES = {349--409},
      ISSN = {0012-7094,1547-7398},
   MRCLASS = {14M25 (14D07 14F40 14J45 32J25)},
  MRNUMBER = {1203231},
MRREVIEWER = {Richard\ M.\ Hain},
       DOI = {10.1215/S0012-7094-93-06917-7},
       URL = {https://doi.org/10.1215/S0012-7094-93-06917-7},
}

@book {Gro,
     TITLE = {Rev\^etements \'etales et groupe fondamental ({SGA} 1)},
    SERIES = {Documents Math\'ematiques (Paris)},
    VOLUME = {3},
      NOTE = {S\'eminaire de g\'eom\'etrie alg\'ebrique du Bois Marie
              1960--61,
              Directed by A. Grothendieck,
              With two papers by M. Raynaud,
              Updated and annotated reprint of the 1971 original [LNM 224, Springer, Berlin]},
 PUBLISHER = {Soci\'et\'e{} Math\'ematique de France, Paris},
      YEAR = {2003},
     PAGES = {xviii+327},
}

@article {Hi,
    AUTHOR = {Hironaka, Heisuke},
     TITLE = {Resolution of singularities of an algebraic variety over a
              field of characteristic zero. {I}, {II}},
   JOURNAL = {Ann. of Math. (2)},
  FJOURNAL = {Annals of Mathematics. Second Series},
    VOLUME = {79},
      YEAR = {1964},
     PAGES = {109--203; 205--326},
}

@article{Gri,
    AUTHOR = {Griffiths, Phillip A.},
     TITLE = {Periods of integrals on algebraic manifolds: {S}ummary of main
              results and discussion of open problems},
   JOURNAL = {Bull. Amer. Math. Soc.},
  FJOURNAL = {Bulletin of the American Mathematical Society},
    VOLUME = {76},
      YEAR = {1970},
     PAGES = {228--296},
      ISSN = {0002-9904},
   MRCLASS = {14.01},
  MRNUMBER = {258824},
MRREVIEWER = {D.\ Lieberman},
       DOI = {10.1090/S0002-9904-1970-12444-2},
       URL = {https://doi.org/10.1090/S0002-9904-1970-12444-2},
}

@article {Sa,
    AUTHOR = {Saito, Morihiko},
     TITLE = {Admissible normal functions},
   JOURNAL = {J. Algebraic Geom.},
  FJOURNAL = {Journal of Algebraic Geometry},
    VOLUME = {5},
      YEAR = {1996},
    NUMBER = {2},
     PAGES = {235--276},
      ISSN = {1056-3911,1534-7486},
   MRCLASS = {14K30 (14C30 14D07 32G20)},
  MRNUMBER = {1374710},
MRREVIEWER = {Claire\ Voisin},
}

@article {Gr,
	 AUTHOR = {Griffiths, Phillip A.},
	  TITLE = {Periods of integrals on algebraic manifolds, {I}, {II}},
	JOURNAL = {Amer. J. Math},
  FJOURNAL = {American Journal of Mathematics},
    VOLUME = {90},
      YEAR = {1968},
     PAGES = {568-626; 805-865},
}

@article {CDK,
    AUTHOR = {Cattani, Eduardo and Deligne, Pierre and Kaplan, Aroldo},
     TITLE = {On the locus of {H}odge classes},
   JOURNAL = {J. Amer. Math. Soc.},
  FJOURNAL = {Journal of the American Mathematical Society},
    VOLUME = {8},
      YEAR = {1995},
    NUMBER = {2},
     PAGES = {483--506},
      ISSN = {0894-0347,1088-6834},
   MRCLASS = {14D07 (14C30 32G20 32J25)},
  MRNUMBER = {1273413},
MRREVIEWER = {Claire\ Voisin},
       DOI = {10.2307/2152824},
       URL = {https://doi-org.libproxy.washu.edu/10.2307/2152824},
}

@article {BPS,
    AUTHOR = {Brosnan, Patrick and Pearlstein, Gregory and Schnell,
              Christian},
     TITLE = {The locus of {H}odge classes in an admissible variation of
              mixed {H}odge structure},
   JOURNAL = {C. R. Math. Acad. Sci. Paris},
  FJOURNAL = {Comptes Rendus Math\'ematique. Acad\'emie des Sciences. Paris},
    VOLUME = {348},
      YEAR = {2010},
    NUMBER = {11-12},
     PAGES = {657--660},
      ISSN = {1631-073X,1778-3569},
   MRCLASS = {14D07 (32G20 32J25)},
  MRNUMBER = {2652492},
MRREVIEWER = {Jan\ Nagel},
       DOI = {10.1016/j.crma.2010.04.002},
       URL = {https://doi-org.libproxy.washu.edu/10.1016/j.crma.2010.04.002},
}

@article {BKV,
    AUTHOR = {Bloch, Spencer and Kerr, Matt and Vanhove, Pierre},
     TITLE = {Local mirror symmetry and the sunset {F}eynman integral},
   JOURNAL = {Adv. Theor. Math. Phys.},
  FJOURNAL = {Advances in Theoretical and Mathematical Physics},
    VOLUME = {21},
      YEAR = {2017},
    NUMBER = {6},
     PAGES = {1373--1454},
      ISSN = {1095-0761,1095-0753},
   MRCLASS = {14J33 (14D07 33E30 34C14 81S40)},
  MRNUMBER = {3780269},
MRREVIEWER = {Alan\ Matthew\ Thompson},
       DOI = {10.4310/ATMP.2017.v21.n6.a1},
       URL = {https://doi-org.libproxy.washu.edu/10.4310/ATMP.2017.v21.n6.a1},
}

@article {Bl,
    AUTHOR = {Bloch, Spencer},
     TITLE = {Algebraic cycles and higher {$K$}-theory},
   JOURNAL = {Adv. in Math.},
  FJOURNAL = {Advances in Mathematics},
    VOLUME = {61},
      YEAR = {1986},
    NUMBER = {3},
     PAGES = {267--304},
      ISSN = {0001-8708},
   MRCLASS = {18F25 (11G45 14C35 19D99 19E15)},
  MRNUMBER = {852815},
MRREVIEWER = {L.\ N.\ Vaserstein},
       DOI = {10.1016/0001-8708(86)90081-2},
       URL = {https://doi-org.libproxy.washu.edu/10.1016/0001-8708(86)90081-2},
}
\end{document}